\PassOptionsToPackage{dvipsnames}{xcolor}
\documentclass[leqno]{siamart1116}

\usepackage{adjustbox}
\usepackage{algpseudocode}
\usepackage{amssymb}
\usepackage{bm,bbm}
\usepackage{booktabs}
\usepackage{caption}
\usepackage{color}
\usepackage{derivative}
\usepackage{dsfont}
\usepackage{dutchcal}
\usepackage{enumerate}
\usepackage{enumitem}
\usepackage{geometry}
\usepackage{graphicx}
\usepackage{mydef}
\usepackage{marginnote}
\usepackage{multirow}
\usepackage[normalem]{ulem}
\usepackage[numbers]{natbib}
\usepackage{siunitx}
\usepackage{subcaption}
\usepackage{xspace}

\let\cite=\citet
\setlist{leftmargin=20pt}

\newcommand{\skipremainderofdocument}[1]{}

\begin{document}

\newcommand{\TheTitle}{%
Robust second-order approximation of the compressible Euler equations
  with an arbitrary equation of state}
\newcommand{\TheAuthors}{B. Clayton, J.-L. Guermond, M. Maier, B. Popov, E. Tovar}

\renewcommand{\thefootnote}{\arabic{footnote}}

\headers{Second-order Finite Element approximation}{\TheAuthors} 

\title{{\TheTitle}%
  \thanks{Draft version, \today%
  \funding{This material is based upon work supported in part by the
    National Science Foundation grants DMS-1912847 (MM), DMS-2045636 (MM),
    DMS-2110868 (JLG, BP), by the Air Force Office of Scientific Research,
    USAF, under grant/contract number FA9550-18-1-0397 (JLG, BP), the Army
    Research Office, under grant number W911NF-19-1-0431 (JLG, BP), and the
    U.S. Department of Energy by Lawrence Livermore National Laboratory
    under Contracts B640889, B641173 (JLG, BP).
    ET acknowledges the support from Los Alamos National Laboratory's
    (LANL) Advanced Simulation and Computing Program, Integrated Codes (IC) and Physics \& Engineering
    Models (PEM) sub-programs, operated by Triad National Security, LLC, for
    the National Nuclear Security Administration of U.S. Department of
    Energy (Contract No. 89233218CNA000001). ET also acknowledges support
    from the U.S. Department of Energy's Office of Applied Scientific
    Computing Research (ASCR) and Center for Nonlinear Studies (CNLS) at
    LANL under the Mark Kac Postdoctoral Fellowship in Applied Mathematics.
    The authors acknowledge the Texas Advanced Computing Center (TACC) at
    The University of Texas at Austin for providing HPC resources that have
    contributed to the research results reported within this paper.
    \url{http://www.tacc.utexas.edu}}}}

\author{Bennett~Clayton\footnotemark[2]
  \and Jean-Luc~Guermond\footnotemark[2]
  \and Matthias~Maier\footnotemark[2]
  \and Bojan~Popov\footnotemark[2]
  \and Eric~J.~Tovar\footnotemark[3]}

\maketitle

\renewcommand{\thefootnote}{\fnsymbol{footnote}}

\footnotetext[2]{%
  Department of Mathematics, Texas A\&M University, 3368 TAMU,
  College Station, TX 77843, USA.}
\footnotetext[3]{%
  X Computational Physics, Los Alamos National Laboratory, P.O. Box 1663,
  Los Alamos, NM, 87545, USA.}

\renewcommand{\thefootnote}{\arabic{footnote}}

\begin{abstract}
  This paper is concerned with the approximation of the compressible Euler
  equations supplemented with an arbitrary or tabulated equation of state.
  The proposed approximation technique is robust, formally second-order
  accurate in space, invariant-domain preserving, and works for every
  equation of state, tabulated or analytic, provided the pressure is
  nonnegative. An entropy surrogate functional that grows across shocks is
  proposed. The numerical method is verified with novel analytical
  solutions and then validated with several computational benchmarks seen in the literature.
\end{abstract}



\pagestyle{myheadings}
\thispagestyle{plain}
\markboth%
  {B. Clayton, J.-L. Guermond, M. Maier, B. Popov, E. Tovar}%
  {Robust approximation of Euler Equations with arbitrary equation of
  state}


\section{Introduction}
The objective of this paper is to continue the work started in
\cite{CGP_2022} where we introduced an explicit and invariant-domain
preserving method to approximate the Euler equations equipped with an
equation of state that is either tabulated or is given by an expression
that is so involved that elementary Riemann problems cannot be solved
exactly or efficiently. The method proposed in \citep{CGP_2022}  is
invariant-domain preserving in the sense that it guarantees that the
density and the internal energy of the approximation are positive. A key
ingredient of the method is a local approximation of the equation of state
using a covolume ansatz (\aka Noble-Abel equation of state) from which
upper bounds on the maximum wave speed are derived for every elementary
Riemann problem. The downside is that the method is first-order accurate in
space. In the present paper we propose a technique that increases the
accuracy in space to second-order and preserves the invariant-domain
properties.  We also introduce a functional that can be used as an entropy
surrogate. The functional in question is shown to be increasing across
shocks in the Riemann problem involved in the construction of the local
maximum wave speed. We show how to use this functional to limit the
internal energy from below. This feature is useful for equations of state
that are tabulated or interpolated from experimental data since in this
case no natural notion of entropy is available.

It is sometimes possible, though possibly expensive, to solve Riemann
problems when the equation of state is analytic.  For instance this is done
in \cite[\S1]{Colella_Glaz_JCP_1985}, \cite{Irving_Causon_Toro_IJNMF_1998},
\cite{Quartapelle_Castelletti_Gardone_Quaranta_2003}. This cannot be done
with tabulated equations of state because the information on the pressure
is incomplete. Several attempts to develop methods working with an
arbitrary or tabulated equation of state have been reported in the
literature. One way to do so consists of using approximate Riemann solvers
like those found in \cite{Dukowicz_JCP_1985},
\citep[\S2]{Colella_Glaz_JCP_1985}, \cite{Roe_Pike_1985},
\cite{Pikes_AIAA_1993}, and \cite{LEE2013165}. One can also simplify the
Riemann problem by using flux splitting techniques as those reported in
\cite{Toro_Castro_Lee_JCP_2015}.  We also refer to \cite{SAUREL2007822},
\cite{Banks2010}, \cite{Dumbser_casulli_2016}, \cite{DUMBSER2013141} where
approximation techniques are developed using approximate Riemann solvers
for various equations of state. Some of these techniques guarantee
positivity of the density, but little else is guaranteed in general. The
method introduced in \cite{CGP_2022} is based on a graph viscosity
technique using upper bounds on the maximum wave speed in the Riemann
problem. Instead of using the two-shock approximation of the Riemann
solution, as done in most methods based on approximate solvers, the method
proposed in \citep{CGP_2022} approximates the pressure in the Riemann fan
by the covolume equation of state and subsequently estimates guaranteed
upper bounds on the maximum wave speed. This in turn ensures that the
internal energy in the proposed algorithm is positive in addition to the
density being positive. If it happens that the equation of state is of
covolume type, then the method also preserves the minimum principle on the
specific entropy.  A technique based on similar principles is reported in
\cite{wang2021stiffened} where the authors use a stiffened gas equation of
state to approximate the pressure in the Riemann fan. To the best of our
knowledge, the method proposed in the present paper is among the very first
ones that are provably invariant-domain preserving for complex or tabulated
equations of state and high-order accurate in space.

The paper is organized as follows. In Section~\ref{Sec:Preliminaries}, we
introduce the mathematical model of interest and the corresponding
notation. We also briefly discuss the assumptions we make on the equation
of state and recall results from~\citep{CGP_2022} that are needed for this
work (\ie we recall the first-order approximation of the Euler equations in
\S\ref{Sec:first_order}).  In Section~\ref{Sec:high_order}, we construct a
provisional update that is higher-order accurate in space. This update is
based on a high-order graph viscosity using an entropy commutator and an
activation function. Then, in Section~\ref{Sec:limiting}, we apply a novel
convex limiting technique that corrects the invariant-domain violations of
the provisional higher-order method. The final result is an approximation
technique that is robust, formally second-order accurate in space,
provably invariant-domain preserving, and works for every equation of state
(tabulated or analytic) that satisfies the mild assumptions stated in
\S\ref{SubSec:EOS}. Finally in Section~\ref{Sec:illustrations}, the method
is verified with analytical solutions and published benchmarks and is
validated with experiments. A short conclusion is given in
\S\ref{sec:conclusion}.


\section{Preliminaries}
\label{Sec:Preliminaries}
We formulate the problem and introduce notation in this section. We also
recall essential results from \citep{CGP_2022} for completeness.

\subsection{The Euler equations}
Let $\Dom$ be a bounded polyhedron in $\Real^d$.  Given some initial data
$\bu_0(\bx)\eqq (\rho_0,\bbm_0,E_0)(\bx)$ and initial time $t_0$, we look
for $\bu(\bx,t) := (\rho,\bbm,E)(\bx,t)$ solving the compressible Euler
equations in the weak sense:
\begin{subequations}\label{Euler}
\begin{align}
  & \partial_t \rho + \DIV(\bv \rho) =  0
  & & \text{\ae $t>t_0$, $\bx \in \Dom$},   \label{mass}
  \\
  & \partial_t \bbm + \DIV\big(\bv\otimes \bbm + p(\bu) \polI_d\big) = \bzero
  &  & \text{\ae $t>t_0$, $\bx \in \Dom$},\label{momentum}
  \\
  & \partial_t E  + \DIV\big(\bv(E + p(\bu))\big) = 0
  & & \text{\ae $t>t_0$, $\bx \in \Dom$}. \label{total_energy}
\end{align}%
\end{subequations}
The components of the dependent variable $\bu\eqq (\rho, \bbm, E)\tr \in
\Real^{d+2}$ (considered to be a column vector) are the density, $\rho$,
the momentum, $\bbm$, and the total mechanical energy, $E$. We also
introduce the velocity $\bv(\bu)\eqq \rho^{-1} \bbm$ and the specific internal
energy, $e(\bu) \eqq \rho^{-1}E - \frac12\|\bv(\bu)\|_{\ell^2}^2$. To
simplify the notation later on we introduce the flux $\polf(\bu) := (\bbm,
\bv(\bu)\otimes \bbm + p(\bu)\polI_d, \bv(\bu)(E+p))\tr \in \Real^{(d+2)\times d}$,
where $\polI_d$ is the $d\CROSS d$ identity matrix.

\subsection{Equation of state}\label{SubSec:EOS}
In \eqref{Euler} the function $p(\bu)$ denotes the pressure which we assume
to be given by an \emph{oracle}. This means we assume to have no a priori
knowledge of the function $p(\bu)$ itself apart from some mild structural
assumptions stated below. We only assume that for a given state $\bu$ we
are able to retrieve a pressure $p(\bu)$ in a suitable way, for example by
evaluating an arbitrary analytic expression, or by deriving a value from
tabulated experimental data. More precisely, in the numerical illustrations
reported in \S\ref{Sec:illustrations}, we consider oracles given by
analytic functions for the ideal gas, van der Waals, Jones--Wilkins--Lee
and Mie-Gruneisen equations of state. We also use the~\texttt{SESAME}
database developed at Los Alamos National Laboratory~\citep{lyon1992sesame}
to test the method with experimental tabulated data.

Throughout the paper we assume that the domain of definition  of the oracle
$p(\bu)$ is the set $\calB(b)\subset\Real^{d+2}$ given by
\begin{equation}\label{def_of_calA}
  \calB(b)\eqq \big\{\bu \eqq(\rho,\bbm, E) \in \mathbb{R}^{d+2}
  \st 0<\rho, \ 0 < 1 - b \rho, \; 0< e(\bu) \big\}.
\end{equation}
We henceforth refer to $\calB(b)$ as the admissible set.  One of the
objectives of the paper
is to guarantee that
the approximation is high-order
accurate and leaves $\calB(b)$ invariant.
The inequality $\rho <
\tfrac1b$ found in the definition of $\calB(b)$ is the so called
\emph{maximum compressibility} condition. The constant $b$ can be set to
zero if the user has no a priori knowledge about the maximum compressibility
of the fluid under consideration. We recall, however, that a large class of
analytic equations of state, such as the Noble-Abel, the Mie-Gruneisen
(with the Hugoniot locus as the reference curve), and the
Noble-Abel-Stiffened-Gas equations of state all involve a maximum
compressibility constant. Finally, we assume that the oracle $p(\bu)$
returns a non-negative pressure,
\begin{equation}
  p: \calB(b) \to \Real_{\ge 0}. \label{Pressure_is_nonnegative}
\end{equation}
The assumption can be weakened, but for the sake of
simplicity, we refrain from doing so in this paper. We
leave this extension for future works.

\subsection{First-order time and space approximation}\label{Sec:first_order}
\label{Sec:low_order_scheme}
The time and space discretization proposed in~\cite{CGP_2022} is based
on~\citep{Guermond_Popov_SINUM_2016}. This method is in some sense a
discretization-agnostic generalization of an algorithm introduced
by~\cite[p.~163]{Lax_1954}. We denote by $t^n$ the current time, $n\in
\polN$, and we denote by $\dt$ the current time step size; that is
$t^{n+1}:=t^n+\dt$. Without going into the details of the space
approximation, we assume that the current approximation of $\bu(\cdot,t^n)$
is a collection of states $\{\bsfU_i\upn\}_{i\in\calV}$, where the index
set $\calV$ is used to enumerate all the degrees of freedom of the
approximation, and $\bsfU_i^n$ is in $\Real^{d+2}$ for all $i\in\calV$. The
update at $t^{n+1}$ is obtained as follows:
\begin{equation}
  \label{low_order_scheme}
  \frac{m_i}{\dt}(\bsfU_i\upLnp-\bsfU_i^n) +\!\!\sum_{j\in \calI(i)}
  \polf(\bsfU_j^n)\,\bc_{ij}
  - \!\!\sum_{j\in \calI(i){\setminus}\{i\}} d_{ij}\upLn (\bsfU^n_j -
  \bsfU^n_i)  = \bzero.
\end{equation}
The quantity $m_i$ is called the lumped mass and we assume that $m_i>0$ for
all $i\in\calV$. The vector $\bc_{ij}\in \Real^d$ encodes the space
discretization. The index set $\calI(i)$ is called the local stencil. This
set collects only the degrees of freedom in $\calV$ that interact with $i$
(\ie $j\not\in\calI(i) \Rightarrow \bc_{ij}=\bzero$). We assume that the
$\Real^d$-valued coefficients $\bc_{ij}$ are such that $\sum_{j\in
\calI(i)} \polf(\bsfU_j^n) \bc_{ij}$ approximates
$\DIV(\polf(\bu(\cdot,t^n)))$ at some grid point $\ba_i$, and the
consistency error in space scales optimally with respect to the mesh size
for the considered approximation setting. We require that the method be
conservative; more precisely, we assume that
\begin{align}\label{cijprop}
  \bc_{ij} = - \bc_{ji}\quad \text{and}
  \quad \sum_{j\in \calI(i)} \bc_{ij} = \bzero.
\end{align}
Concrete expressions for $\bc_{ij}$ and $m_i$ are given in
\citep[\S4]{Guermond_Popov_Tomas_CMAME_2019} for continuous and
discontinuous finite elements as well as for finite volumes.  The
computations reported at the end of this paper are done with piecewise
linear continuous finite elements. But to stay general, we continue with
the abstract discretization-agnostic notation introduced above.

For completeness we now recall how the graph viscosity $d_{ij}\upLn$ is
defined in \citep{CGP_2022}. Given $i\in\calV$ and $j\in \calI(i)$, we set
$\bn_{ij}\eqq \bc_{ij} \|\bc_{ij}\|_{\ell^2}^{-1}$.  For $Z\in \{i,j\}$, we
set $\bsfU_Z\eqq (\varrho_Z,\bsfM_Z,\sfE_Z)\tr,$ $\sfp_Z\eqq
p(\bsfU_Z^n)$,
$\sfe_Z\eqq e(\bsfU_Z^n)$,
$\Gamma_Z\eqq\varrho_Z+\frac{\sfp_Z(1-b\varrho_Z)}{\sfe_Z}$, and
$\mathcal{E}_Z \eqq \sfE_Z -
\frac{\|\bsfM_Z-(\bsfM_Z\SCAL\bn_{ij})\bn_{ij}\|_{\ell^2}^2}{2\varrho_Z}$. The
non-negativity assumption on the pressure
\eqref{Pressure_is_nonnegative} and the assumptions on the density ($0<
\varrho_Z< \frac1b$) implies that
\begin{equation}
  \gamma_Z\eqq \frac{\Gamma_Z}{\varrho_Z}\ge 1. \label{def_of_gamma_Z}
\end{equation}
Then we consider the following Riemann problem:
\begin{equation}
  \partial_t
  \begin{pmatrix}
    \rho \\ m \\ \mathcal E \\ \Gamma
  \end{pmatrix}
  + \partial_x
  \begin{pmatrix}
    m\\
    \tfrac{1}{\rho}m^2 + \calpcov \\
    \tfrac{m}{\rho} (\mathcal{E}+ \calpcov )\\
    \tfrac{m}{\rho}\Gamma
  \end{pmatrix}=0,
  \quad \text{with} \quad
  \calpcov(\rho,m,\calE,\Gamma)\eqq
  \frac{\frac{\Gamma}{\rho}-1}{1-b\rho}\left(\mathcal{E} - \tfrac{m^2}{2\rho}\right),
  \label{Ext_Euler}
\end{equation}
with left data $(\rho_i,\bbm_i\SCAL\bn_{ij},\mathcal{E}_i, \Gamma_i)\tr$
and right data $(\rho_j,\bbm_j\SCAL\bn_{ij},\mathcal{E}_j, \Gamma_j)\tr$.
This problem is well-posed because $\gamma_Z\ge 1$. Its complete solution
is given in \cite[\S4]{CGP_2022}. We denote by
$\lambda(\bn_{ij},\bsfU_i,\bsfU_j)$ the maximum wave speed in this
problem. Let $\calA$ be a nontrivial convex subset of $\calB(b)$. We say
that $\calA$ is an invariant set for \eqref{Ext_Euler} if for every pair of
Riemann data in $\calA$ the solution of  \eqref{Ext_Euler} takes values in
$\calA$. We then have:
\begin{theorem}[{\citep[Thm.~4.6]{CGP_2022}}]
  \label{Thm:summary}
  Let $i\in\calV$. Let $\calA \subset\calB(b)$ be a convex invariant set
  for \eqref{Ext_Euler}. Assume that $\bsfU_j^n\in \calA$ for all
  $j\in\calI(i)$. For all $j\in\calI(i)$, let
  $\wlambda(\bn_{ij},\bsfU_i^n,\bsfU_j^n)$ be any positive number larger
  than or equal to $\lambda(\bn_{ij},\bsfU_i^n,\bsfU_j^n)$. Let
  \begin{equation}
    \label{def_dij_n}
    d_{ij}\upLn \eqq \max(\wlambda(\bn_{ij},\bsfU_i^n,\bsfU_j^n)
    \|\bc_{ij}\|_{\ell^2}, \wlambda(\bn_{ji},\bsfU_j^n,\bsfU_i^n)
    \|\bc_{ji}\|_{\ell^2}).
  \end{equation}
  Assume that $\dt$ is small enough so that
  $\dt \sum_{j\in\calI(i){\setminus}\{i\}} \frac{2d_{ij}\upLn}{m_i}\le 1$.
  Let $\bsfU_i\upLnp$ be the update defined in
  \eqref{low_order_scheme}. Then
 $\bsfU\upLnp_i \in \calA \subset\calB(b)$.
\end{theorem}
\begin{remark}[Invariant domains]
  Theorem~\ref{Thm:summary} asserts that every convex invariant set in
  $\calB(b)$ is invariant by the update procedure \eqref{low_order_scheme}
  with the artificial viscosity $d_{ij}\upLn$ defined in \eqref{def_dij_n}.
  This means that if $\bsfU\upn_i\in\calA$ for all $i\in\calV$, then
  $\bsfU\upLnp_i\in\calA$ for all $i\in\calV$. We say that the method is
  invariant-domain preserving \emph{for $\calA$}. Notice in particular that
  the method is invariant-domain preserving for the admissible set
  $\calB(b)$. The admissible set $\calB(b)$ may not be the smallest
  invariant domain. For instance, if the oracle admits a mathematical
  entropy $s$, then the approximation defined above is also
  invariant-domain preserving for the set $\calA(\bu_0)\eqq
  \{\bv\in\calB(b)\st s(\bv)\ge \essinf_{\bx\in \Dom} s(\bu_0(\bx))\}$; see
  \eg \citep[Cor.~4.2]{Guermond_Popov_SINUM_2016}. By slight abuse of
  terminology and provided the context is unambiguous we will simply call a
  method invariant-domain preserving without quantifying the precise convex
  set.
\end{remark}
A source code to compute $\wlambda(\bn_{ij},\bsfU_i^n,\bsfU_j^n)$ is
available online \citep{guermond_jean_luc_2021_4685868}. A notable drawback
of the graph viscosity \eqref{def_dij_n} is that it reduces the space
accuracy of the method to first order. The remainder of the paper is
concerned with constructing a higher-order approximation and a respective
novel convex limiting technique that is invariant-domain preserving.

\begin{remark}[Pressure approximation]
  Notice that the oracle is only invoked to compute the left and right
  values of the pressure in \eqref{Ext_Euler}.  The pressure in the Riemann
  fan is approximated by two covolume equations of state. There is one
  covolume equation of state for each side of the contact discontinuity
  since $\gamma=\gamma_L$ on the left of the contact wave and
  $\gamma=\gamma_R$ on the right.
\end{remark}


\section{Provisional higher-order method}
\label{Sec:high_order}
In this section, we introduce a provisional higher-order method that
extends the update~\eqref{low_order_scheme} to second order when using
linear finite elements for the discretization. A convex limiting technique
for this provisional update is introduced in \S\ref{Sec:limiting}. The
provisional method is based on the work reported
in~\citep{Guermond_Nazarov_Popov_Tomas_SISC_2019,maier2021efficient}.

\subsection{The method}
Higher-order accuracy in space requires using the consistent mass
matrix instead of the lumped mass matrix for the discretization of the
time derivative. By reducing dispersive errors, the consistent mass
matrix is known to yield superconvergence at the grid points; see \eg
\cite{FLD:FLD719}, \citep{Guermond_pasquetti_2013},
\citep[Sec.~3.4]{maier2021efficient}, \cite{Thompson_2016}.  Let
$\polM$ be the mass matrix with entries $\{m_{ij}\}_{i,j\in\calV}$,
then $\polM^{-1}$ can be approximated by $\polI+\polB$ where
$b_{ij} = \delta_{ij} - \frac{m_{ij}}{m_j}$, $\delta_{ij}$ is the
Kronecker symbol, and $m_i\eqq \sum_{j\in\calI(i)} m_{ij}$. This
approximation has been shown in \citep{FLD:FLD719,Guermond_pasquetti_2013}
to be superconvergent for piecewise-linear continuous finite elements.  Let
$\sfR\in \Real^{I}$ with $I\eqq\text{card}(\calV)$, then using this
approximation we have $(\polM^{-1} \sfR)_i\approx \sfR_i +
\sum_{j\in\calI(i)} \left(b_{ij} \sfR_j - b_{ji} \sfR_i\right)$.  Notice
that $\sum_{j\in\calI(i)} b_{ji} \sfR_i = 0$, because $\sum_{j\in\calI(i)}
b_{ji}=0$. The skew-symmetry of the summand in $\sum_{j\in\calI(i)}
\left(b_{ij} \sfR_j - b_{ji} \sfR_i\right)$ is used in \S\ref{Sec:limiting}
for limiting purposes.

Let $\bsfU^{\text{H},n+1}$ denote the high-order update (here the
superindex $\upH$ reminds us that the update is higher order accurate). Then,
for every $i\in\calV$, the provisional high-order update is given by:
\begin{subequations}
  \label{provisional_high_order}
  \begin{align}
    \frac{m_i}{\tau}\left(\bsfU_i^{\text{H},n+1}-\bsfU_i^n\right)
    = \bR_i^n + \sum_{j\in\calI(i)}\left(b_{ij}\bR_j^n - b_{ji}\bR_i^n\right),
    \\
    \text{with}\qquad \qquad  \bR_i^n  \eqq\sum_{j\in\calI(i)}
    \left(-\polf(\bsfU_j^n)\SCAL \bc_{ij} + d_{ij}\upHn(\bsfU_j^n -
    \bsfU_i^n)\right).
    \label{High_order_RHS}
  \end{align}
\end{subequations}
The high-order graph viscosity coefficient $d_{ij}\upHn$, defined in
Section~\ref{entropy_commutator},  shares the same properties as its
low-order counterpart which are necessary for conservation:
\begin{equation}
  d_{ij}\upHn = d_{ji}\upHn, \quad
  0\le d_{ij}\upHn\leq d_{ij}\upLn,
  \quad \text{for every }j\in\calI(i),\; i\in\calV.
\end{equation}

\subsection{Entropy commutator}\label{entropy_commutator}
Using the entropy-viscosity methodology introduced in~\cite{guermond2011entropy},
we now discuss the construction of the high-order graph viscosity
coefficient $d_{ij}\upHn$ that is used in the high-order update
\eqref{High_order_RHS}. A complicating factor in this construction is
that the pressure is given by an oracle. We thus do not have
a priori knowledge on the equation of state and entropies of the
system.

Recall that $(\eta(\bu),\bF(\bu))$ is an entropy pair for the Euler
equations if
\begin{equation}
  \DIV\bF(\bu) = (\GRAD_\bu\eta(\bu))\tr\DIV\polf(\bu),\qquad
  \forall \bu\in\calB(b),\label{chain_rule}
\end{equation}
where $\polf(\bu)$ is the flux of the system~\eqref{Euler}.  Since one may
not have access to entropies for tabulated equations of state, we are going
to use at every $i\in\calV$ and every time $t^n$ one entropy pair
associated with the following flux
\begin{equation}
  \label{local_entropy}
  \polf^{i,n}(\bu) \eqq
  \begin{pmatrix}
    \bbm \\
    \bv\otimes\bbm + \calpcov^{i,n}(\bu)\polI_d \\
    \bv(E + \calpcov^{i,n}(\bu))
  \end{pmatrix},
  \qquad \calpcov^{i,n}(\bu) \eqq (\gamma_i^{\min,n} - 1)\frac{\rho
  e(\bu)}{1 - b\rho}.
\end{equation}
Here, $\gamma_i^{\min,n} \eqq\min_{j\in\calI(i)} \gamma_j^{n}$ with
$\gamma_j^{n}\eqq 1+ \frac{\sfp_j^n(1-b\varrho_j^n)}{\varrho_j^n \sfe_j^n}$
(see \eqref{def_of_gamma_Z}).  We use the following shifted Harten entropy pair in
the numerical tests reported in
\S\ref{Sec:illustrations}:
\begin{equation}
  \label{local_harten}
  \eta^{i,n}(\bu) \eqq \left(\frac{\rho^2 e(\bu)}{(1 - b \rho)^{1 - \gamma_i^{\min,n}}}
    \right)^{\frac{1}{\gamma_i^{\min,n} + 1}} - \frac{\rho}{\varrho_i^n}\eta_{\text{ref}}^{i,n},
  \qquad
  \bF^{i,n}(\bu) \eqq \frac{\bbm}{\rho} \eta^{i,n}(\bu),
\end{equation}
where
$\eta_{\text{ref}}^{i,n}\eqq
\left(\frac{(\varrho_i^n)^2e(\bu_i^n)}{(1 - b \varrho_i^n)^{1 -
      \gamma_i^{\min,n}}}\right)^{\frac{1}{\gamma_i^{\min,n} + 1}} $.
Then, we estimate ``entropy production'' by inserting the approximate
solution $\bu_h^n$ into a discrete counterpart of~\eqref{chain_rule}
which we write as follows:
\begin{equation}
  \label{local_chain_rule}
  \int_\Dom\left(\DIV\bF^{i,n}(\bu)
  -(\GRAD_\bu\eta^{i,n}(\bu))\tr\DIV\,\polf^{i,n}(\bu)\right)\varphi_i
  \diff x =0,
\end{equation}
where $\{\varphi_i\}_{i\in\calV}$ are the shape functions of the
finite element approximation.  Notice that the above identity holds
true for every smooth function $\bu\in \calB(b)$ and for every
$i\in\calV$. Substituting
$\bu_h\eqq\sum_{i\in\calV}\bsfU_i^n\varphi_i$
into~\eqref{local_chain_rule}, we estimate the local entropy residual
as follows:
\begin{equation}
  \label{local_entropy_residual}
  N_i^n\eqq \sum_{j\in\calI(i)}\Big(\bF^{i,n}(\bsfU_j^n)
  - (\GRAD_{\bu}\eta^{i,n}(\bsfU_i^n))\tr\polf^{i,n}(\bsfU_j^n)\Big)\SCAL\bc_{ij}.
\end{equation}
The residual~\eqref{local_entropy_residual} can be thought of as a
measure of how well the discrete solution $\bu_h^n$
satisfies~\eqref{local_chain_rule} in each local stencil
$\calI(i)$. We then define the normalized entropy residual as follows:
\begin{subequations}
  \begin{align}
    \label{normalized_residual}
    R_i^n & \eqq \frac{\abs{N_i^n}}{D_i + \epsilon D^{\max}},\qquad
    D^{\max} = \max_{i\in\calV} D_i, \qquad  \epsilon = 10^{-2},
    \\
    D_i^n & \eqq \Big|\sum_{j\in\calI(i)}\bF^{i,n}(\bsfU_j^n)\SCAL\bc_{ij}\Big|
    + \Big| \sum_{j\in\calI(i)}(\GRAD_{\bu}\eta^{i,n}(\bsfU_i^n))
    \tr \polf^{i,n}(\bsfU_j^n)\SCAL\bc_{ij}\Big|.
  \end{align}
\end{subequations}
By definition, the normalized residual $R_i^n$ has values in $[0,1]$
and, for piecewise linear finite elements, behaves like $\calO(h)$
where $h$ is the typical meshsize. Finally we set
\begin{equation}
  d_{ij}\upHn\eqq d_{ij}\upLn \max(\psi(R_i^n),\psi(R_j^n)),
\end{equation}
where the activation function $\psi$ is defined as follows:
 \begin{equation}
    \psi(x) \eqq \frac{4x_0^3 - (x+x_0)(x-2x_0)\big\{(x-2x_0) -
    \text{ReLU}(x-2x_0)\big\}}{4x_0^3},
  \end{equation}
%
and $\text{ReLU}(x) = (x+|x|)/2$ is the rectified linear activation
function. Notice that $\psi(0)=0$, $\psi(x_0)=\frac12$ and $\psi(x)=1$ for
all $x\in [2x_0,1]$ for a chosen $x_0\in[0,0.5]$. As a result, one recovers
$d_{ij}\upHn = d_{ij}\upLn $ if the entropy residual is larger than $2x_0$.
Note that $\psi(x^*)=x^*$ for $x^*=x_0(\frac32 - \frac12 \sqrt{9 - 16x_0})$
(the identity $\psi(x^*)=x^*$ also holds for $x^*\in\{0,1\}$). When $x\in
[0,x^*]$ we have $\psi(x)\sim \tfrac34 \frac{x^2}{x_0^2} + \calO(x^3)$.
The numerical tests reported in \S\ref{Sec:illustrations} are done with
$x_0=0.4$ (\ie $x^*\approx 0.27751$). An activation function with the same
purpose is used in \cite[Eq.~(8)]{persson_peraire_2006}. Up to a
translation, the activation function therein behaves like
$\frac12+\frac12\sin(\pi\frac{x-x_0}{2x_0})$ in the interval $[0,2x_0]$.

\begin{remark}[Entropy shift] The entropy shift considered in
    \eqref{local_harten} is motivated by the observation that the
    numerator $N_i^n$ is unchanged by the change of variable
    $\eta(\bu) \mapsto \eta(\bu) -\lambda \rho$ for all
    $\lambda\in \Real$ and for every entropy $\eta$ associated with
    the flux~\eqref{local_entropy}. The constant $\lambda$ is chosen
    in \eqref{local_harten} so that $\eta^{i,n}(\bu_i^n)=0$. This
    entropy shift  was first is introduced in
    \citep[\S3.4]{Guermond_Nazarov_Popov_Tomas_SISC_2019}.
\end{remark}


\section{Convex Limiting}\label{Sec:limiting}
The high-order update \eqref{provisional_high_order} is not guaranteed to
be invariant-domain preserving; in particular, it is not guaranteed to stay
in the admissible set $\calB(b)$. In order to correct this defect we now
discuss a new convex limiting technique that re-establishes invariant-domain
preservation for the final high-order update, \ie
$\bsfU\upnp_i\in\calB(b)$.

\subsection{Key observation}
A key observation is that one can
rewrite~\eqref{low_order_scheme} as follows:
\begin{subequations}
  \begin{align}
    \label{def_dij_scheme_convex}
    \bsfU_i\upLnp =
    \bigg(1 -\!\! \sum_{j \in \calI(i)\backslash\{i\}} \frac{2 \dt
    d_{ij}\upLn}{m_i} \bigg) \bsfU_{i}^{n}
    + \!\! \sum_{j \in \calI(i)\backslash\{i\}} \frac{2 \dt d_{ij}\upLn}{m_i}
    \overline{\bsfU}_{ij}\upn, \\
    \label{bar_states} \text{with} \qquad
    \oline{\bsfU}_{ij}^n\eqq\frac12(\bsfU_i^n + \bsfU_j^n)
    - \frac{1}{2 d_{ij}\upLn}\left(\polf(\bsfU_j^n) - \polf(\bsfU_i^n
    )\right)\SCAL\bc_{ij}.
  \end{align}
\end{subequations}
That is, the low-order update $\bsfU_i\upLnp$ is a convex combination
(under the appropriate CFL condition) of the local state $\bsfU_i^n$ and
the auxiliary states $\{\oline{\bsfU}_{ij}^n\}_{j\in\calV}$, \ie
\begin{align}
  \label{eq:convex_combination}
  \bsfU_i\upLnp\in\conv\{\overline\bsfU_{ij}^n\st j\in\calI(i)\}.
\end{align}
The main result established in \citep[Thm.~4.6]{CGP_2022} (and summarized
in Theorem~\ref{Thm:summary}) is that under the CFL condition stated in
Theorem~\ref{Thm:summary} and the definition \eqref{def_dij_n} for
$d_{ij}\upLn$, the states $\{\oline{\bsfU}_{ij}^n\}_{j\in\calI(i)}$ are in
$\calB(b)$ provided this is already the case of the states
$\{\bsfU_j^n\}_{j\in\calI(i)}$. This is done by proving that the states
$\oline{\bsfU}_{ij}^n$ are space averages of the solution to a Riemann
problem. We refer the reader to
\citep[Sec.~3.3]{Guermond_Popov_SINUM_2016},
\citep[Sec.~3.2]{Guermond_Nazarov_Popov_Tomas_SISC_2019} and
\citep[Sec.~3.2]{Guermond_Popov_Tomas_CMAME_2019} where this is discussed
in detail. We are going to use the states $\oline{\bsfU}_{ij}^n$ to define
local bounds in space and time to perform the limiting of the high-order
states $\{\bsfU\upHnp_i\}_{i\in\calV}$.

\subsection{Entropy surrogate}
\label{Sec:entropy_surogate}
Our goal is to use the methodology introduced
in~\citep{guermond2011entropy} to perform the convex limiting of the update
$\bsfU_i\upHnp$. In this context the use of an oracle
with little a priori knowledge on the equation of state poses a
significant challenge as it makes it impossible to properly define an
entropy. We resolve the impasse by introducing an artificial
\emph{surrogate} entropy that has the right mathematical properties for the
convex limiting methodology to be applied; see
Theorem~\ref{Thm:surrogate_entropy}.

For any admissible state $\bu\in\calB(b)$ and $\gamma\ge1$, we define
\begin{equation}
  S(\bu,\gamma) := \frac{(\rho\,e)(\bu)}
  {\rho^\gamma}(1 - b\rho)^{\gamma-1},
\end{equation}
where $\rho(\bu)$ and $e(\bu)$ are the density and specific internal
energy of the state $\bu$. Furthermore, for every index $i\in\calV$, we
set
\begin{equation}
  \label{def_of_gamma_in}
  \gamma_i^n \eqq 1 + \frac{\sfp_i^n(1-b\varrho_i^n)}{\varrho_i^n\sfe_i^n},
  \qquad\Gamma_i^n \eqq \varrho_i^n\,\gamma_i^n,
\end{equation}
where $\varrho_i^n\eqq \rho(\bsfU_i^n)$, $\sfp_i^n\eqq p(\bsfU_i^n)$
and
$\sfe_i^n \eqq e(\bsfU_i^n)= \frac{1}{\varrho_i^n}\big(\sfE_i^n -
\frac{\|\bsfM_i^n\|_{\ell^2}^2}{2\varrho_i^n}\big)$. The following
result is the key motivation for the definition of an entropy
surrogate.
\begin{lemma}
  \label{Lem:entropy_increase}
 For all $i\in\calV$,  assume that $\bsfU_i^n\in \calB(b)$. For all $i\in\calV$,
 all $j\in\calI(i)$, all $\gamma_{ij}\in[1,\min(\gamma_i^n,\gamma_j^n)]$, 
all left data
  $(\varrho_i^n,\bsfM_i^n\SCAL\bn_{ij},\mathcal{E}_i^n, \Gamma_i^n)\tr$ and
  right data $(\varrho_j^n,\bsfM_j^n\SCAL\bn_{ij},\mathcal{E}_j^n,
  \Gamma_j^n)\tr$ in the
  extended Riemann problem \eqref{Ext_Euler},
  and all $\bu\in\calB(b)$,
  we set
  \begin{equation}
    \label{Lem:surrogate_entropy_functional}
    \Psi_{ij}(\bu) := \rho e(\bu) - S_{ij}^{\min}
    \rho^{\gamma_{ij}}(1 - b\rho)^{1 - \gamma_{ij}},
    \quad\text{where}\;
    S_{ij}^{\min}\eqq
    \min\big(S(\bsfU_i^n,\gamma_{ij}),S(\bsfU_j^n,\gamma_{ij})\big).
  \end{equation}
  Then, $\Psi_{ij}(\bu)$ increases across shocks in the solution of the
  extended Riemann problem \eqref{Ext_Euler}  (if a shock wave exists).
\end{lemma}

\begin{proof}
  We omit the superscript $^n$ in the proof to simplify the notation. The
  solution to the extended Riemann problem \eqref{Ext_Euler}, $(\rho, m,
  \mathcal E, \Gamma)\tr(x,t)$, is given in \cite{CGP_2022}.  In
  particular, we have that $\gamma(x,t) = \gamma_i$ if $x/t <
  v^*$ and $\gamma(x,t) = \gamma_j$ if $x/t > v^*$.  Here,
  $\gamma(x,t)=\Gamma(x,t)/\rho(x,t)$, and $v^*$ is the speed of the contact
  wave.  Let $Z \in\{i, j\}$, with the convention that the index $i$ is for the
  left state and $j$ is for the right state. Assume that the $Z$-wave is a
  shock wave.  Let $\bu_Z \in \calB(b)$ be the state before the shock and
  let $\bu \in \calB(b)$ be an arbitrary state connected to $\bu_Z$ through
  a shock curve. With the notation $\tau\eqq\frac{1}{\rho}$ for the
  specific volume (not to be confused with the time step), and since
  $\gamma=\gamma_Z$ along the wave curve, the Rankine-Hugoniot condition
  implies that (see \cite[Eq.~(4.8), p.~144]{Godlewski_raviart_1996}),
  $e(\bu) - \sfe_Z + \tfrac12 \big(\calpcov(\tau, e(\bu)) + \calpcov(\tau_Z,
  \sfe_Z)\big)(\tau -\tau_Z) = 0$, where by slight abuse of notation
  we renamed the pressure in \eqref{Ext_Euler} by setting $\calpcov(\tau,
  e)\eqq \frac{\gamma_Z-1}{\tau-b} e$.
  We then infer that $e(\bu)$ only depends on $\rho$ along the shock curve; more precisely, we have
  \begin{equation*}
    e(\bu) =  \sfe_Z\,
    \frac{1 -\frac{(\gamma_Z-1)(\tau-\tau_Z)}{2(\tau_Z -b)}}
    {1+\frac{(\gamma_Z-1)(\tau-\tau_Z)}{2(\tau-b)}} \qqe r(\tau).
  \end{equation*}
  Notice that this function is well defined only on the interval
    $(\tau_Z^\infty,\infty)$ with
 \begin{equation}
 \tau_Z^\infty\eqq\tfrac{(\gamma_Z-1)\tau_Z + 2b}{\gamma_Z+1}, \label{def_of_tauZinfty}
 \end{equation}
    and it is nonegative on the interval $(\tau_Z^\infty,\tau_Z^0)$
    with $\tau_Z^0\eqq\tfrac{(\gamma_Z+1)\tau_Z - 2b}{\gamma_Z-1}$.
    Notice that $b<\tau_Z^\infty < \tau_Z < \tau_Z^0$ since we assumed that $\bsfU_Z\in\calB(b)$.
  We now show that the function
  $\calB(b) \ni \bu \mapsto \rho e(\bu) - c \rho^{\gamma} (1 - b\rho)^{1 -
    \gamma}$ is nonnegative and increasing on the shock curve for all
  $\gamma\in (1,\gamma_Z]$ and
  $c \in (0,\sfe_Z (\tau_Z - b)^{\gamma-1}]$. This will then prove the
  assertion for the choices $\gamma\eqq\gamma_{ij}$ and
  $c\eqq S_{ij}^{\min}$, because $\gamma_{ij}\in (1, \gamma_Z]$ and
\begin{equation*}
    0 < S_{ij}^{\min} \le \varrho_Z^{1-\gamma_{ij}} \sfe_Z (1
    -b\varrho_Z)^{\gamma_{ij}-1} = \sfe_Z (\tau_Z -b)^{\gamma-1},
\end{equation*}
\ie $S_{ij}^{\min} \in (0,\sfe_Z (\tau_Z - b)^{\gamma-1}]$.
Setting $q(\tau) \eqq \frac{\gamma_Z-1}{\tau-b}
r(\tau)$, we have $q'(\tau)=- \frac{4\gamma_Z(\gamma_Z-1)}{(\tau + \tau_Z - 2b
+ \gamma(\tau-\tau_Z))^2}<0$; hence, the pressure, $q$, is a monotone increasing
function of $\rho$ along shock curves. By definition of shock curves,
starting from the state $\bsfU_Z$ the pressure increases, we conclude that
$\rho$ also increases along shock curves, \ie $\rho\in [\varrho_Z,\frac{1}{b})$
or $\tau\in (b,\tau_Z]$. Actually, the pressure is finite only in the range
$\tau\in (\tau_Z^\infty,\tau_Z]$; hence, showing that $\calB(b) \ni \bu \mapsto
\rho e(\bu) - c \rho^{\gamma} (1 - b\rho)^{1 - \gamma}$ is nonnegative
increasing on shock curves, is equivalent to showing that
  \begin{equation*}
    (\tau_Z^\infty,\tau_Z]
    \ni \tau \mapsto g(\tau)\eqq \tau^{-1} r(\tau) - c \tau^{-1}
    (\tau - b)^{1-\gamma}
  \end{equation*}
  is nonnegative decreasing on shock curves. Recall that the specific
  entropy for the covolume equation of state is given by $s(\tau,\sfe) =
  \log(\sfe^{\frac1{\gamma_Z - 1}}(\tau - b))$.  A fundamental property of the
  specific entropy is that it is an increasing function along shocks.  That
  is, $s(\tau,r(\tau))$ is a decreasing function over the interval $\tau
  \in (\tau_Z^\infty,\tau_Z]$.  This also means that $(\tau_Z^\infty,\tau_Z] \ni \tau \mapsto
  \varsigma(\tau) \eqq \exp\big((\gamma_Z - 1)s(\tau, r(\tau))\big) =
  r(\tau)(\tau - b)^{\gamma_Z - 1}$ is a decreasing function. We now have
  \begin{equation*}
   g(\tau) = \tau^{-1} r(\tau) - c \tau^{-1} (\tau-b)^{1-\gamma} =
  \varsigma(\tau) \tau^{-1} (\tau-b)^{1-\gamma_Z} - c \tau^{-1}
  (\tau-b)^{1-\gamma}.
  \end{equation*}
  Let $\widetilde{\varsigma}$ be the
  defined by $\widetilde{\varsigma}(\tau) \eqq
  \varsigma(\tau)(\tau-b)^{\gamma-\gamma_Z}$.  Then we have
  \begin{equation*}
    g(\tau)= \frac{(\tau-b)^{1-\gamma}}{\tau}
    (\widetilde{\varsigma}(\tau)-c).
  \end{equation*}
  Computing the derivative of $g$, we see,
  \begin{equation*}
    g'(\tau) =
    \frac{(\tau-b)^{1-\gamma}}{\tau}\widetilde{\varsigma}'(\tau) -
    \frac{(\gamma - 1)\tau + (\tau - b)}{\tau^2 (\tau - b)^\gamma}
    (\widetilde{\varsigma}(\tau) - c).
  \end{equation*}
  Note that
  \begin{equation*}
    \widetilde{\varsigma}'(\tau) = (\tau
    - b)^{\gamma - \gamma_Z} \varsigma'(\tau) + (\gamma - \gamma_Z) (\tau
    - b)^{\gamma - \gamma_Z - 1} \varsigma(\tau).
  \end{equation*}
  Then $\widetilde{\varsigma}'(\tau) \leq 0$ as $\varsigma'(\tau) \leq 0$
  for $\tau\in(\tau_Z^\infty,\tau_Z]$, $\gamma \leq \gamma_Z$, and
  $0\le \varsigma(\tau)$, and
  $\inf_{\tau \in (b,\tau_Z]} \widetilde{\varsigma}(\tau) =
  \widetilde{\varsigma}(\tau_Z) = \sfe_Z(\tau_Z - b)^{\gamma - 1} \ge
  c$ as $\widetilde{\varsigma}$ is a decreasing function.  Thus, by
  the choice of the constants $c$ and $\gamma$, we have that
  $g(\tau) \geq 0$ and $g'(\tau) \leq 0$. Hence,
  $[\varrho_Z,\frac{1}{\tau_Z^\infty})\ni \rho \mapsto \rho e(\bu) - c
  \rho^{\gamma} (1 - b\rho)^{1 - \gamma}$ is nonnegative increasing on
  shock curves. This completes the proof.
\end{proof}

We now define the surrogate entropy. For all $i\in\calV$, we set
\begin{subequations}
  \begin{align}
    \gamma_i^{\min,n}&\eqq \min_{j\in\calI(i)} \gamma_j^n,\qquad
    S_i^{\min,n}\eqq \min(\min_{j\in\calI(i)}S(\bsfU^n_j; \gamma_i^{\min,n}),
    \min_{j\in\calI(i)}S(\overline{\bsfU}^n_{ij};\gamma^{\min,n}_i)), \label{entropy_min}\\
    \label{surrogate_entropy_functional}
    \Psi^s_i(\bu) &\eqq \rho e(\bu) - S_i^{\min,n}
    \rho^{\gamma_i^{\min,n}}(1 - b\rho)^{1 - \gamma_i^{\min,n}}.
  \end{align}
\end{subequations}
The following result summarizes the content of this section.
\begin{theorem}[Surrogate entropy]
  \label{Thm:surrogate_entropy} The following
  holds true with the same assumptions and
  definitions as in Lemma~\ref{Lem:entropy_increase}:
  \begin{enumerate}[font=\upshape,label=(\roman*)]
    \item
      \label{item1:Thm:surrogate_entropy}
      The function $\Psi^s_i: \calB(b) \to \Real$ is concave.
    \item
      \label{item2:Thm:surrogate_entropy}
      Let $\bsfU_i\upLnp$ be the update defined in
      \eqref{low_order_scheme}. We have $\Psi^s_i(\bsfU_i\upLnp)\ge 0$
      under the CFL condition $\dt \sum_{j\in\calI(i){\setminus}\{i\}}
      \frac{2d_{ij}\upLn}{m_i}\le 1$.
    \item
      \label{item3:Thm:surrogate_entropy}
      Let $j\in\calI(i)$. Consider the extended Riemann problem
      \eqref{Ext_Euler} with left data
      $(\varrho_i^n,\bsfM_i^n\SCAL\bn_{ij},\mathcal{E}_i^n,
      \Gamma_i^n)\tr$ and right data
      $(\varrho_j^n,\bsfM_j^n\SCAL\bn_{ij},\mathcal{E}_j^n,
      \Gamma_j^n)\tr$. If the solution has shock waves, then the
      function $\Psi^s_i(\bu)$ increases across the shocks.
  \end{enumerate}
\end{theorem}
\begin{proof}
  \ref{item1:Thm:surrogate_entropy} Recall that $\bu \mapsto \rho e(\bu)$
  is concave.  Moreover, the function $\phi(x) = x(\frac{1}{x}
  -b)^{\gamma_i^{\min,n}}$ is convex because $\gamma_i^{\min,n}\ge 1$. As a
  result, $- S_i^{\min,n} \rho^{\gamma_i^{\min,n}}(1 - b\rho)^{1 -
  \gamma_i^{\min,n}} = - S_i^{\min,n} \phi(\rho)$ is concave because $
  S_i^{\min,n}\ge 0$. This proves that $\Psi^s_i: \calB(b) \to \Real$ is
  concave.\\
  \ref{item2:Thm:surrogate_entropy} Recalling from
  Theorem~\ref{Thm:summary} that under the CFL condition $\dt
  \sum_{j\in\calI(i){\setminus}\{i\}} \frac{2d_{ij}\upLn}{m_i}\le 1$, the
  state $\bsfU_i\upLnp$ is in the convex hull of the states
  $\{\overline{\bsfU}_{ij}^n\}_{j\in\calI(i)}$, the concavity of the
  functional $\Psi^s_i$ implies that $\Psi^s_i(\bsfU_i\upLnp)\ge
  \min_{j\in\calI(i)} \Psi^s_i(\overline{\bsfU}_{ij}^n)$. But for all
  $j\in\calI(i)$ we have
  \begin{align*}
    \Psi^s_i(\overline{\bsfU}_{ij}^n) & = (\rho
    e)(\bu)(\overline{\bsfU}_{ij}^n) - S_i^{\min,n}
    (\overline{\varrho}_{ij}^n)^{\gamma_i^{\min,n}}(1 -
    b\overline{\varrho}_{ij}^n)^{1 - \gamma_i^{\min,n}} \\ & = (
    S(\overline{\bsfU}_{ij}^n, \gamma_i^{\min,n}) - S_i^{\min,n})
    (\overline{\varrho}_{ij}^n)^{\gamma_i^{\min,n}}(1 -
    b\overline{\varrho}_{ij}^n)^{1 - \gamma_i^{\min,n}}.
  \end{align*}
  And we conclude that $\Psi^s_i(\overline{\bsfU}_{ij}^n)\ge 0$ by
  definition of $S_i^{\min,n}$. The assertion readily follows. \\
  \ref{item3:Thm:surrogate_entropy} The assertion is a consequence of
  Lemma~\ref{Lem:entropy_increase}.
\end{proof}

\subsection{Limiting on the density}
Limiting on the density is performed exactly as done in
\citep[\S4.4]{Guermond_Nazarov_Popov_Tomas_SISC_2019}. First, we
define the local bounds
\begin{align}
  \varrho^{\min,n}_i \eqq \min_{j\in\calI(i)} \overline{\varrho}^n_{ij}
  \quad \text{ and } \quad \varrho^{\max,n}_i \eqq \max_{j\in\calI(i)}
  \overline{\varrho}^n_{ij},
\end{align}
where $\overline{\varrho}^n_{ij}$ is the density of the auxiliary state
$\overline{\bsfU}^n_{ij}$.

Second, we relax these bounds to ensure that second-order accuracy is
maintained in the maximum norm.  The relaxation is done as in
\citep[\S4.7]{Guermond_Nazarov_Popov_Tomas_SISC_2019} with a
modification of $\varrho^{\max,n}_i$ to accommodate the covolume
constraint $1-b\rho>0$.  This modification is justified in the
following lemma.

\begin{lemma}[Maximum density bound]
  \label{Lem:density_upper_bound} The following holds true.
  \begin{enumerate}[font=\upshape,label=(\roman*)]
    \item
      \label{Item1:Lem:density_upper_bound}
      The density in the Riemann problem~\eqref{Ext_Euler} satisfies the
      following upper bound
\begin{equation}
  \rho\,\leq\,\max_{Z\in\{i,j\}} \frac{1}{\tau_Z^\infty} =
    \max_{Z\in\{i,j\}}\frac{(\gamma_Z+1)\rho_Z}{(\gamma_Z-1) + 2b\rho_Z}.
\end{equation}
    \item
      \label{Item2:Lem:density_upper_bound}
      Under the CFL condition stated in Theorem~\ref{Thm:summary}, the
      low-order update satisfies the following:
\begin{equation}
        \varrho_i\upLnp \le
        \frac{(\gamma_i^{\min,n}+1)\varrho_i^{\max,n}}{(\gamma_i^{\min,n}-1) + 2b\varrho^{\max,n}}.
        \label{Rho_max_b_not_0}
\end{equation}
      %
  \end{enumerate}
\end{lemma}
\begin{proof}
  \ref{Item1:Lem:density_upper_bound} Let $Z\in\{i,j\}$, with the
  convention that the index $i$ is for the left state and $j$ is for
  the right state. Recall that the pressure in the Riemann
    problem~\eqref{Ext_Euler} is defined by the function
    $\calpcov(\bsfU)$ (with a slight abuse of notation) and we also
    have $\sfp_Z=\calpcov(\sfU_Z)$.  If the elementary $Z$-wave is an
  expansion then $\rho$ decreases along the expansion wave and we have
  $\rho\le \varrho_Z$. If instead the elementary $Z$-wave is a shock,
  we have established in the proof of Lemma~\ref{Lem:entropy_increase}
  that $\tau\in(\tau_Z^\infty,\tau_Z]$, \ie
  $\rho \in [\rho_Z,\frac{1}{\tau_Z^\infty})$; see
  \eqref{def_of_tauZinfty}. Whence the assertion.
  \\
  \ref{Item2:Lem:density_upper_bound}
  Using \eqref{def_dij_scheme_convex}, we observe that
\begin{align*}
    \varrho_i\upLnp &\le \max_{j\in\calI(i)} \oline{\varrho}_{ij}^n
                      \le \max_{j\in\calI(i)} \frac{\gamma_j+1}{(\gamma_j-1)\tau_j + 2b}
     = \max_{j\in\calI(i)} \frac{(\gamma_j+1)\varrho_j}{(\gamma_j-1) + 2b\varrho_j}.
\end{align*}
As the function $\frac{(\gamma+1)\rho}{(\gamma-1) + 2b\rho}$ is monotone increasing with respect to
$\rho$ and monotone decreasing with respect to $\gamma$, the inequality
\eqref{Rho_max_b_not_0} follows readily.
\end{proof}

Proceeding as in \citep[\S4.7]{Guermond_Nazarov_Popov_Tomas_SISC_2019}, we
estimate the local curvature of the density by
\begin{equation}
  \Delta^2 \varrho^n_i := \frac{\sum_{j\in\calI(i){\setminus} \{i\}}
    \beta_{ij}(\varrho^n_i - \varrho^n_j)}{\sum_{j\in\calI(i){\setminus} \{i\}}\beta_{ij}},\qquad
  \overline{\Delta^2 \varrho^n_i} := \frac{1}{2(\text{card}(\calI(i))-1)}
  \sum_{j\in\calI(i){\setminus} \{i\}} (\tfrac12
  \Delta^2\varrho^n_i + \tfrac12 \Delta^2 \varrho^n_j),
\end{equation}
where $\beta_{ij}=\int_\Dom \GRAD \varphi_j\SCAL\GRAD\varphi_i\diff x$
are the stiffness coefficients of the Laplace operator and we recall
that $\{\varphi_i\}_{i\in\calV}$ are the global shape shape functions.
We note in passing that
$\sum_{j\in\calI(i){\setminus} \{i\}}\beta_{ij} = - \beta_{ii} =
-\int_\Dom (\GRAD \varphi_i)^2\dx \ne 0$.  The relaxed local density
bounds are then defined as follows:
\begin{align}
  \overline{\varrho^{\min,n}_i}  &:= \max( \varrho^{\min,n}_i -
  \overline{\Delta^2 \varrho^n_i}, (1 - r_h) \varrho^{\min,n}_i),
  \\
  \overline{\varrho^{\max,n}_i}  &:=
     \min(\varrho^{\max,n}_i + \overline{\Delta^2 \varrho^n_i},
      \tfrac{(1+\gamma_i^{\min,n})\varrho^{\max,n}_i}{\gamma_i^{\min,n}-1+ 2b\varrho^{\max,n}_i},
      \, (1 + r_h) \varrho^{\max,n}_i).
\end{align}
where $r_h := (\frac{m_i}{|D|})^{\frac{1.5}{d}}$. The argumentation for the
presence of the terms involving $1-r_h$ and $1+r_h$ is given in Remark~4.15
in \citep[\S4.7]{Guermond_Nazarov_Popov_Tomas_SISC_2019}.

The actual limiting on the density guaranteeing that
$\overline{\varrho^{\min,n}_i} \le \varrho_i^{n+1} \le
\overline{\varrho^{\max,n}_i}$ is done as explained in
\citep[\S4.4]{Guermond_Nazarov_Popov_Tomas_SISC_2019}.

\subsection{Limiting on the surrogate entropy}
\label{Sec:convex_limiting}
After limiting the density, we limit the surrogate entropy. Recall
that we established in Theorem~\ref{Thm:surrogate_entropy} that
$\Psi^s_i(\bsfU_i\upLnp) \ge 0$. We want the limiting operation to
guarantee that $\Psi^s_i(\bsfU_i^{n+1}) \ge 0$ as well. This limiting
in turn implies a positive lower bound on the internal energy. In this
section we use the notation $\bv\eqq (\rho(\bv),\bbm(\bv),E(\bv))$ for
all $\bv\in\Real^{d+2}$; that is, $\rho(\bv)$ is the first component of $\bv$,
  $\bbm(\bv)\in \Real^d$ is composed of the components $2$ to $d+1$, and $E(\bv)$ is the last component.

The limiting on $\Psi_i^s$ is done with the convex limiting explained
in \citep[Sec.~4.2]{Guermond_Nazarov_Popov_Tomas_SISC_2019}.  Given
some state $\bsfP_{ij}\in \Real^{d+2}$,
$j\in\calI(i){\setminus}\{i\}$, and $\ell_0\in [0,1]$ so that
$0<\rho(\bsfU_i\upLnp+ \ell_0 \bsfP_{ij})<b$, one has to find the
largest $\ell$ in $[0,\ell_0]$ so that
$\Psi_i^s(\bsfU_i\upLnp+ \ell \bsfP_{ij})\ge 0$. Notice that
$\bsfU_i\upLnp+ \ell \bsfP_{ij} \in \calA(b)\eqq \{\bv\in
\Real^{d+2}\st 0<\rho(\bv)<b^{-1}\}$ for all $\ell\in [0,\ell_0]$
because $\calA(b)$ is convex. This implies that
$\Psi_i^s(\bsfU_i\upLnp+ \ell \bsfP_{ij})$ is well-defined for all
$\ell\in [0,\ell_0]$.
One sets $\ell=\ell_0$ if
$\Psi_i^s(\bsfU_i\upLnp+ \ell_0\bsfP_{ij})\ge 0$.  Otherwise, setting
$g(\ell)\eqq \Psi_i^s(\bsfU_i\upLnp+ \ell \bsfP_{ij})$, one solves the
equation $g(\ell)=0$. This equation has at least one solution because
$\Psi_i^s$ is continuous.  The solution set is connected because
$\Psi_i^s$ is concave.  The solution set is a singleton if $\Psi_i^s$
is strictly concave.

\begin{lemma}[Internal energy] Let $\bsfU^{n+1}$ be the final
    stage obtained after limiting on the density and the surrogate
    entropy. Then the specific internal energy of this state satisfies the following lower
    bound for all $i\in\calV$:
  \begin{equation}
    \sfe_i^{n+1} \ge S_i^{\min,n} \big(\tfrac{1}{\varrho_i^n}-b\big)^{1-\gamma_i^{\min,n}}\ge 0.
    \end{equation}
  \end{lemma}

\begin{remark}[Entropy for the Covolume EOS]
  If the oracle coincides with the covolume equation of state, then
  limiting the entropy surrogate is equivalent to limiting the
  physical entropy.
\end{remark}

We now propose two ways to find $\ell^*$ so that $g(\ell^*)=0$
assuming that this root is unique. Both methods are iterative and are
guaranteed to return an answer $\tilde\ell$ that is such that
$\tilde\ell\uparrow \ell^*$ (hence $g(\tilde\ell)\ge 0$ for every
termination criteria).
The first method consists of proceeding as described in full details in
Section~7.5.4 in \cite{Guermond_Popov_Tomas_CMAME_2019} (see also the end
of Section~4.6 in \citep{Guermond_Nazarov_Popov_Tomas_SISC_2019}). The line
search is based on the Newton-Secant method. It uses the secant method on
the left of $\ell^*$ and the Newton method of the right. The convergence
rate is between $1$ and $2$.
The second method is based on the quadratic Newton method and its
convergence rate is cubic. Instead of solving
$\Psi_i^s(\bsfU_i\upLnp+ \ell\bsfP_{ij})= 0$ one defines
$\Phi_i^s(\bv)\eqq \rho(\bv) \Psi^s_i(\bv)$ and
solves
$\Phi_i^s(\bsfU_i\upLnp+\ell\bsfP_{ij})= 0$.  The solution sets of
the two equations $\Psi_i^s(\bsfU_i\upLnp+ \ell\bsfP_{ij})= 0$ and
$\Phi_i^s(\bv)\eqq \rho(\bv) \Psi^s_i(\bv)$
are identical since we assumed
that $0<\rho(\bsfU_i\upLnp+ \ell\bsfP_{ij})$ for all
$\ell\in[0,\ell_0]$. A key to the method is the following result.
\begin{lemma} \label{Lem:psi_is_concave} Let
  $\bu\in\calA(b)\eqq \{\bv\in \Real^{d+2}\st 0<\rho(\bv)<b^{-1}\}$
  and $\bp\in\Real^{d+2}$. Assume that $\bu+\ell_0\bp\in\calA(b)$. Let
  $f:[0,\ell_0] \ni \ell \mapsto f(\ell)\eqq \Phi_i^s(\bu + \ell
  \bp)$. The sign of $f'''(l)$ is constant over $[0,\ell_0]$
(\ie  $f''(l)$ is monotone  over $[0,\ell_0]$).
\end{lemma}

\begin{proof} By definition we have
  $\Phi_i^s(\bv) = \rho(\bv) E(\bv)-\frac12\|\bbm(\bv)\|_{\ell}^2 - c
  \rho(\bv)^{\gamma+1}(1-b\rho(\bv))^{1-\gamma}$, with $c\eqq
  S_{i}^{\min,n}$ and $\gamma\eqq \gamma_i^{\min,n}$. Notice that
  $\bu+\ell\bp\in\calA(b)$ for all $\ell\in[0,\ell_0]$ since $\calA(b)$ is
  convex; as a result, $f(\ell)$ is well-defined for all
  $\ell\in[0,\ell_0]$. A direct computation shows that
  \begin{align*}
    f'''(\ell) =(\rho(\bp))^3 \partial_{\rho}^3 \Phi_i^s(\bu + \ell
    \bp).
  \end{align*}
  But
  \begin{align*}
    \partial_{\rho}^3 \Phi_i^s(\bv) =
    -c\gamma(\gamma^2-1)\rho(\bv)^{\gamma-2}(1-b\rho(\bv))^{-\gamma-2}.
  \end{align*}
  Hence $f'''(\ell)$ has the same sign as $-\rho(\bp)$ for all
  $\ell\in[0,\ell_0]$.
\end{proof}

We now show how the quadratic Newton algorithm can be implemented to
estimate $\limiter^*$ from below. We initialize the iterative process by
setting $\limiter_L=0$ and $\limiter_R=\limiter_0$. Then, let $\limiter_L$
and $\limiter_R$ be the current left and right estimates of $\limiter^*$
with $0\le \limiter_L<\limiter^*<\limiter_R \le \limiter_0$. We now
construct a quadratic polynomial $P_L(\limiter)$ such that
$P_L(\limiter_L)=f(\limiter_L)$, $P_L(\limiter_R)= f(\limiter_R)$ and
$P_L'(\limiter_L)=f'(\limiter_L)$. We similarly define the quadratic
polynomial $P_R(\limiter)$ such that $P_R(\limiter_L)=f(\limiter_L)$,
$P_R(\limiter_R)= f(\limiter_R)$ and $P_L'(\limiter_R)=f'(\limiter_R)$.
Using the divided difference notation we have:
\begin{subequations}
  \begin{align}
    P_L(\limiter) &:= f(\limiter_L) +
    f[\limiter_L,\limiter_L](\limiter-\limiter_L) +
    f[\limiter_L,\limiter_L,\limiter_R](\limiter-\limiter_L)^2,
    \\
    P_R(\limiter) &:= f(\limiter_R) +
    f[\limiter_R,\limiter_R](\limiter-\limiter_R) +
    f[\limiter_L,\limiter_R,\limiter_R](\limiter-\limiter_R)^2.
  \end{align}
\end{subequations}
\begin{lemma}
  \label{Lem:roots}
  The following holds true:
  \begin{enumerate}[font=\upshape,label=(\roman*)]
    \item
      The polynomials $P_L(\limiter)$ and $P_R(\limiter)$ bound the
      function $f(\limiter)$ in the following sense:
      \begin{equation}
        \min\!\big(P_L(\limiter),P_R(\limiter)\big)
        < f(\limiter)<
        \max\!\big(P_L(\limiter),P_R(\limiter)\big),
        \qquad \forall \limiter\in (\limiter_L,\limiter_R).
        \label{Eq:quad_zeros}
      \end{equation}
    \item
      $P_L(\limiter)$ and $P_R(\limiter)$ have each a unique zero over the
      interval $(\limiter_L,\limiter_R)$ respectively given by
      \begin{subequations}\label{roots}
        \begin{align}
          \limiter^L(\limiter_L,\limiter_R) &\eqq \limiter_L -
          \frac{2f(\limiter_L)}{f'(\limiter_L)+\sqrt{f'(\limiter_L)^2
          -4f(\limiter_L) f[\limiter_L,\limiter_L,\limiter_R]}},
          \label{roots_a}
          \\
          \limiter^R(\limiter_L,\limiter_R)&\eqq \limiter_R -
          \frac{2f(\limiter_R)}{f'(\limiter_R)+\sqrt{f'(\limiter_R)^2
          -4f(\limiter_R) f[\limiter_L,\limiter_R,\limiter_R]}}.
          \label{roots_b}
        \end{align}
      \end{subequations}
    \item
      Properties (i) and (ii) imply that
      \begin{equation}
        \min\!\big(\limiter^L(\limiter_L,\limiter_R),\limiter^R(\limiter_L,\limiter_R)\big)
        < \limiter^*<
        \max\!\big(\limiter^L(\limiter_L,\limiter_R),\limiter^R(\limiter_L,\limiter_R)\big).
        \label{Eq:quad_zeros_bound}
      \end{equation}
  \end{enumerate}
\end{lemma}
\begin{proof}
  (i) The argumentation is largely the same as in the proof Lemma~4.5 in
  \citep{Guermond_Popov_Fast_Riemann_2016} and relies on the sign of
  $f'''(\limiter)$ being constant over $(\limiter_L,\limiter_R)$, as
  established in Lemma~\ref{Lem:psi_is_concave}, which implies that one of
  the quadratic polynomials ($P_L(\limiter)$ or $P_R(\limiter)$) is above
  $f(\limiter)$ and the other one is below $f(\limiter)$ for $\limiter\in
  (\limiter_L,\limiter_R)$.
  (ii) Moreover, both polynomials have exactly one zero in the interval
  $(\limiter_L,\limiter_R)$ given by \eqref{roots} even in the degenerated
  case when one or both of them are linear functions.
\end{proof}
Lemma~\ref{Lem:roots} implies that
$\min\big(\limiter^L(\limiter_L,\limiter_R),\limiter^R(\limiter_L,\limiter_R)\big)$
is always a lower bound on the root $\limiter^*$. The quadratic Newton
algorithm consists of replacing $\limiter_L$ by
$\min\big(\limiter^L(\limiter_L,\limiter_R),\limiter^R(\limiter_L,\limiter_R)\big)$
and $\limiter_R$ by $\max\big(\limiter^L(\limiter_L,\limiter_R),
\limiter^R(\limiter_L,\limiter_R)\big)$ and looping until some threshold
criterion is reached. The convergence rate of this algorithm is cubic.


\section{Numerical illustrations}
\label{Sec:illustrations}
In this section we illustrate the proposed method with several benchmarks
and experiments.

\subsection{Preliminaries}
Two independent codes have been written to ascertain reproducibility. The
first one, henceforth referred to as the \texttt{TAMU} code, does not use
any particular software and is written in Fortran 95/2003. It is based on
Lagrange elements on simplices and is dimension-independent. The
\texttt{TAMU} code is used only for the one-dimensional tests. The second
code is based on continuous $\polQ_1$ finite elements on quadrangular
meshes and use the \texttt{Ryujin}~\citep{maier2021efficient,ryujin-2021-b}
software, a high-performance finite-element solver based on the
\texttt{deal.II}. The \texttt{Ryujin} code is used for all two-dimensional tests.

The time-stepping in both codes is done with three stage, third-order
strong stability preserving Runge-Kutta method, SSPRK(3,3). The time step
is defined by the expression $\dt = \frac12\text{CFL}\CROSS \min_{i\in
\calI} \frac{m_i}{\sum_{j\in\calI(i){\setminus}\{i\}}d_{ij}\upLn}$ with
$d_{ij}\upLn$ defined in \eqref{def_dij_n} and $\text{CFL}\in (0,1)$ is
fixed by the user. We refer the reader to~\citep[Sec.~4]{GKMPT_NS_2022} for
a discussion on the implementation of boundary conditions for the Euler
Equations. In this work, we have modified the non-reflecting boundary
conditions described in~\citep[Sec.~4.3.2]{GKMPT_NS_2022} to account for an
arbitrary of state; for brevity, we skip the discussion of the details of
said boundary conditions. All the computations involving dimensional
quantities are done in the \texttt{SI} unit system unless otherwise
specified.

The numerical Schlieren shown in the numerical illustrations is computed with the discrete version of~\citep[Eq.~(35)]{banks_2008} with $\beta = 15$; see the Supplemental Material of~\cite{Guermond_Popov_Tomas_CMAME_2019} for more details.

\subsection{Equations of state}
\label{eos_defs}
In this section we list the equations of state that we use in the numerical
illustrations below. In all the tests reported, the equations of state are
used as an oracle. We make no assumptions on the physical validity of the
equations of state and only require that they provide a positive pressure.

\subsubsection*{Noble-Abel equation of state} The caloric Noble-Abel (or
covolume) equation of state reads
\begin{equation}
  \label{covolume}
  p(\rho, e) \eqq (\gamma-1) \frac{\rho e }{1 - b\rho}.
\end{equation}

\subsubsection*{Van der Waals equation of state} The caloric Van der Waals
equation of state is given by
\begin{equation}
  \label{vdw_eqn} p(\rho, e)
  \eqq (\gamma-1) \frac{\rho e + a \rho^2}{1 - b\rho} - a
  \rho^2.
\end{equation}

\subsubsection*{Mie-Gruneisen with linear Hugoniot locus}
The Mie-Gruneisen equation of state with a linear Hugoniot locus as the
reference curve is defined by
\begin{subequations}
  \label{MG_eos}
  \begin{align}
    p(\rho, e) &\eqq p_{\text{ref}}(\rho) + \rho\Gamma(\rho)\big(e - e_{\text{ref}}(\rho)\big), \\
    \text{with}\quad   p_{\text{ref}}(\rho) &\eqq P_0 + \tilde{\rho}_0 c_0^2\frac{1 - \frac{\tilde{\rho}_0}{\rho}}{\Big(1 - s(1 - \frac{\tilde{\rho}_0}{\rho})\Big)^2}, \quad
    e_{\text{ref}}(\rho) \eqq e_0 + \frac{P_0 + p_{\text{ref}}(\rho)}{2\tilde{\rho}_0}(1 - \frac{\tilde{\rho}_0}{\rho}).
  \end{align}
\end{subequations}
We refer the reader to~\cite[Sec.~4.4]{Menikoff2007} for a discussion of
this particular equation of state and respective parameters. For
simplicity, we take $\Gamma(\rho)\eqq \Gamma_0$ and $P_0 = 0,\,e_0 = 0$.

\subsubsection*{Jones-Wilkins-Lee equation of state}
The pressure given by the Jones-Wilkins-Lee (JWL) equation of state is
defined as follows:
\begin{equation}
  \label{jwl_eqn}
  p(\rho, e) \eqq A\Big(1 -
  \frac{\omega}{R_1} \frac{\rho}{\rho_0} \Big) \exp\Big( -R_1
  \frac{\rho_0}{\rho} \Big) + B \Big(1 - \frac{\omega}{R_2}
  \frac{\rho}{\rho_0} \Big) \exp\Big( -R_2
  \frac{\rho_0}{\rho} \Big) + \omega \rho e.
\end{equation}
The JWL equation of state was first introduced
in~\cite[Eq.~(I.3)]{lee_1968}. We refer the reader
to~\cite{segletes2018examination} for a discussion of the various forms of
this equation of state seen in the literature. We note that~\eqref{jwl_eqn}
can be recast in ``Mie-Gruneisen'' form as follows:
\begin{align*}
    p(\rho, e) &\eqq p_{\text{ref}}(\rho) +\omega \rho \big( e - e_{\text{ref}}(\rho)\big), \\
  \text{with}\quad  p_{\text{ref}}(\rho) &\eqq A e^{-R_1\frac{\rho_0}{\rho}} + B e^{-R_2\frac{\rho_0}{\rho}},
 \quad e_{\text{ref}}(\rho) \eqq \frac{A}{R_1}\frac{\rho}{\rho_0}e^{-R_1\frac{\rho_0}{\rho}} + \frac{B}{R_2}\frac{\rho}{\rho_0}e^{-R_2\frac{\rho_0}{\rho}}.
\end{align*}

\subsection{Convergence tests}
We now verify the accuracy of the proposed method. We define a
consolidated error indicator at time $t$ by accumulating the relative
error in the $L^q$-norm ($q\in[1,\infty]$):
\begin{equation}
  \delta_q(t):=\frac{\|\rho_h(t)-\rho(t)\|_{q}}{\|\rho(t)\|_{q}}
  + \frac{\|\bbm_h(t)-\bbm(t)\|_{q}}{\|\bbm(t)\|_{q}} \\
  + \frac{\|E_h(t)-E(t)\|_{q}}{\|E(t)\|_{q}},
  \label{eq:delta}
\end{equation}
where $\rho(t),\,\bbm(t),\,E(t)$ are the exact states at time $t$, and
$\rho_h(t),\,\bbm_h(t),\,E_h(t)$ are the finite element approximations
at time $t$ for the respective conserved variables.

\subsubsection{1D -- Smooth traveling wave}%
We consider a one-dimensional test proposed
in~\citep[Sec.~5.2]{Guermond_Nazarov_Popov_Tomas_SISC_2019} consisting of a
smooth traveling wave. The goal of this test is to show that we achieve (at
least) second-order accuracy in space with any equation of state used for
the pressure oracle. The smooth traveling wave is an exact solution to the
Euler equations where the primitive variables are set as follows:
\begin{subequations}
  \label{1D_smooth_wave}
  \begin{align}
    \rho(x,t) &=
      \begin{cases}
          \rho_0 + 2^6(x_1 - x_0)^{-6}(x - v_0 t - x_0)^3 (x_1 - x + v_0 t)^3 \quad
          & \text{if } x_0 \leq x - v_0 t \leq x_1,
          \\
          \rho_0
          & \text{otherwise},
      \end{cases}\\
     v(x, t) &= v_0, \qquad
     p(x,t) = p_0,
  \end{align}
\end{subequations}
where $x_0$ and $x_1$ are arbitrary constants such that $x_0 < x_1$. Just
as in~\citep{Guermond_Nazarov_Popov_Tomas_SISC_2019}, we set the constants
to $x_0=0.1$ and $x_1=0.3$. Notice that by fixing the pressure to be
constant, the solution is independent of the equation of state. The
internal energy is initiated by using the respective equations of state
defined in \S\ref{eos_defs}. The constants $p_0,\,\rho_0,\,v_0$ are
chosen to accommodate for the material in question.

  \begin{table}[ht]
  \centering
  \begin{tabular}[b]{lcrcrcrcr}
      \toprule
      & \multicolumn{2}{c}{Ideal} & \multicolumn{2}{c}{VdW}
      & \multicolumn{2}{c}{JWL}   & \multicolumn{2}{c}{MG} \\
      \cmidrule(lr){2-3} \cmidrule(lr){4-5} \cmidrule(lr){6-7} \cmidrule(lr){8-9}
      $|\calV|$ & $\delta_\infty(T)$ & & $\delta_\infty(T)$ &
                & $\delta_\infty(T)$ & & $\delta_\infty(T)$ &
      \\[0.3em]
      101  & 1.94e-02 &   -- & 1.24e-01 &   -- & 7.93e-02 &   -- & 1.24e-05 &  --  \\
      201  & 4.03e-03 & 2.27 & 6.24e-03 & 4.30 & 2.53e-02 & 1.65 & 2.56e-06 & 2.28 \\
      401  & 7.91e-04 & 2.35 & 9.92e-04 & 2.65 & 3.61e-03 & 2.81 & 5.03e-07 & 2.35 \\
      801  & 1.44e-04 & 2.46 & 1.75e-04 & 2.51 & 1.31e-04 & 4.78 & 9.17e-08 & 2.46 \\
      1601 & 2.75e-05 & 2.39 & 3.29e-05 & 2.41 & 2.51e-05 & 2.38 & 1.75e-08 & 2.39 \\
      3201 & 5.18e-06 & 2.41 & 6.17e-06 & 2.41 & 4.73e-06 & 2.41 & 3.29e-09 & 2.41 \\
      6401 & 9.69e-07 & 2.42 & 1.16e-06 & 2.42 & 8.87e-07 & 2.42 & 6.22e-10 & 2.41 \\
      \bottomrule
  \end{tabular}
  \caption{%
    $\delta_\infty(T)$ error \eqref{eq:delta} and corresponding convergence
    rates for the one-dimensional smooth traveling wave problem with exact
    solution~\eqref{1D_smooth_wave} under uniform refinement of the
    interval $D=(0,1)$.}
  \label{tab:smooth_wave}
\end{table}

We consider the following four configurations (here, $T$ is the final time
of the simulation):
\begin{itemize}[noitemsep]
  \item
    Ideal EOS:
    \eqref{covolume} with $b =0$, $\gamma=1.4$, $\rho_0 = p_0 = v_0 =
    1$, $T =0.6$;
  \item
    Van der Waals EOS:
    \eqref{vdw_eqn} with $a = 1$, $b = 0.075$,  $\gamma=1.4$, $\rho_0 =
    p_0 = v_0 = 1$, $T =0.6$;
  \item
    Jones-Wilkins-Lee EOS:
    \eqref{jwl_eqn} with $A = 1$, $B = -1$, $R_1 = 2$,
    $R_2 = \omega = \rho_0 = p_0 = v_0 = 1$, $T =0.6$;
  \item
    Mie-Gruneisen EOS: \eqref{MG_eos} with $\tilde{\rho}_0 = 2790$, $c_0 = 5330$, $s = 1.34$, $\Gamma_0 = 2.00$,
    $\rho_0 = 3500$, $p_0 = \num{1e11}$, $v_0 = \num{1e4}$, $T =
    \num{6e-5}$.
  \end{itemize}

The tests are performed on uniform meshes with the domain $D = (0,1)$ using
the \texttt{TAMU} code. The first mesh is composed of 100 cells. The
other meshes are obtained by uniform refinement via bisection. We use
$\text{CFL}=0.1$ and set Dirichlet boundary conditions for all tests. We
report in Table~\ref{tab:smooth_wave} the quantity $\delta_\infty(T)$ and
the respective convergence rates for the equations of state used. We
observe that the convergence rate is greater than 2 with each EOS (this is
a well known super convergence effect observed on uniform meshes
\citep{FLD:FLD719,Guermond_pasquetti_2013,Thompson_2016}).

\subsubsection{2D -- Isentropic Vortex with Van der Waals EOS}
To demonstrate higher-order convergence in $\polR^2$, we consider a novel
exact solution of the Euler equations~\eqref{mass}--\eqref{total_energy}
using the Van der Waals equation of state~\eqref{vdw_eqn} with $b\eqq 0$.
The exact solution is a modified version of the isentropic vortex problem
with an ideal gas equation of state (see~\cite{YEE1999199}). For
completeness, a derivation of the solution is presented in
Appendix~\ref{appendix:isen_vortex_vdw}.

\begin{table}[ht]
  \centering
  \begin{tabular}[b]{lcrcrcr}
    \toprule
    $|\calV|$ & $\delta_1(T)$ & & $\delta_2(T)$ & & $\delta_\infty(T)$ &
    \\[0.3em]
    \cmidrule(lr){2-3} \cmidrule(lr){4-5} \cmidrule(lr){6-7}
    289     & 1.17e-01 &   -- & 2.01e-01 &  --  & 6.82e-01 &   -- \\
    1089    & 1.18e-02 & 3.46 & 2.65e-02 & 3.06 & 1.05e-01 & 2.82 \\
    4225    & 7.92e-04 & 3.98 & 1.96e-03 & 3.84 & 7.87e-03 & 3.82 \\
    16641   & 5.57e-05 & 3.87 & 1.32e-04 & 3.93 & 5.50e-04 & 3.88 \\
    66049   & 5.07e-06 & 3.48 & 1.20e-05 & 3.48 & 7.79e-05 & 2.83 \\
    263169  & 7.55e-07 & 2.76 & 2.25e-06 & 2.42 & 2.03e-05 & 1.95 \\
    1050625 & 1.64e-07 & 2.20 & 5.51e-07 & 2.04 & 5.52e-06 & 1.88 \\
    4198401 & 4.08e-08 & 2.01 & 1.38e-07 & 2.00 & 1.51e-06 & 1.87 \\
    \bottomrule
  \end{tabular}
  \caption{%
    Error quantity~\eqref{eq:delta} and convergence rates for the isentropic
    vortex problem with the Van der Waals equation of state. The exact solution
    is given by \eqref{isentropic_vortex}.}
    \label{tab:isentropic_vdw}
\end{table}

Recalling that $a$ and $\gamma$ are the parameters of the Van der Waals
equation of state~\eqref{vdw_eqn}, the isentropic vortex solution is given
by
\begin{subequations}
  \label{isentropic_vortex}
  \begin{align}
    \rho(\bx,t)
    & = \Big[ \frac{3C}{8a} - \frac{1}{2}\sqrt{\frac{9C^2}{16a^2}
      + \frac{2}{a}\big(F + \frac{1}{2r_0^2} \psi(\overline{\bx})^2 \big)} \Big]^2, \\
    \bv(\bx,t)  & = \bv_\infty + \psi(\overline{\bx})\big(-\bar{x}_2, \bar{x}_1\big), \\
    p(\bx,t) & = C(\gamma-1)\rho(\bx,t)^\gamma - a
    \rho(\bx,t)^2,
  \end{align}
\end{subequations}
where $\overline{\bx} := \bx - \bx^0 - \bv_\infty t = (\overline{x}_1,
\overline{x}_2)$, $C \eqq (p_\infty + a
\rho_\infty^2)/\rho_\infty^{3/2}$, $F \eqq -a\rho_\infty -
3p_\infty/\rho_\infty$. Here $\rho_\infty$ and $p_\infty$ are the density
and pressure in the far field, and
$\psi(\bx) := \frac{\beta}{2\pi} \exp\Big(\frac12 (1 - \frac{1}{r_0^2}
  \|\bx\|^2_{\ell^2})\Big).$
We set the far field conditions to $\rho_\infty = 0.1$, $p_\infty = 1$ and
$\bv_\infty = (1, 1)$. We also set $\gamma = \frac32$ and $a = 1$. This gives
$C = \frac{101}{\sqrt{10}}$ and $F= -\frac{301}{10}$. The rest of the constants
are set as follows:
$\bx^0 = (-1,-1)$, $r_0 = 1$, $\beta = 20$.
We perform the convergence tests using the \texttt{Ryujin} code. The
computational domain is set to $D = (-5,5)\CROSS(-5,5)$. We run the
simulations until the final time $T = 2$ with $\text{CFL} = 0.1$ and
Dirichlet boundary conditions. In Table~\ref{tab:isentropic_vdw}, we report
the $\delta_q(T)$ errors and the respective convergence rates for
$q\in\{1,2,\infty\}$. The tests are done over eight quadrilateral grids and
seven levels of uniform refinements. We observe second order convergence in
all the norms.

\subsection{Benchmark configurations}
In this section, we consider common
benchmarks seen in the literature and modify them appropriately using
different equations of state. Our goal is to compare the present method
with the state of the art in the literature.
\subsubsection{1D -- Woodward-Colella Blast Wave}
We demonstrate the robustness of the proposed method by reproducing
the Woodward-Colella interacting blast wave benchmark using the
Jones-Wilkins-Lee equation of state~\eqref{jwl_eqn}.
For this benchmark, we consider two different cases.
The first is that seen in~\cite[Sec.~5.1]{Toro_Castro_Lee_JCP_2015}
and the second consists of parameters found in~\cite[Tab.~2(``HMX'')]{lee_1968}.
The parameters for both cases are given in Table~\ref{tab:jwl_cases}.

\begin{table}[ht]
  \centering
  \begin{tabular}[b]{llllllll}
      \toprule
      & $A$ & $B$ & $R_1$ & $R_2$ & $\omega$ & $\rho_0$ & Final time $T$
      \\[0.3em]
      \textbf{Case 1} & \SI{6.321e3}{}   & \SI{-4.472}{}     & 11.3 & 1.13 & 0.8938 & \SI{1}{}    & \SI{0.038}{} \\
      \textbf{Case 2} & \SI{7.7828e11}{} & 7.071428$\times10^9$ &  4.2 & 1.00 & 0.3000 & \SI{1891}{} & \SI{8.2e-4}{}\\
      \bottomrule
    \end{tabular}
  \caption{JWL parameters for Woodward-Colella interacting blast wave benchmark.}
  \label{tab:jwl_cases}
\end{table}
\begin{figure}[ht]
     \centering
         \includegraphics[trim={50 0 50 13},clip,width=0.35\linewidth]{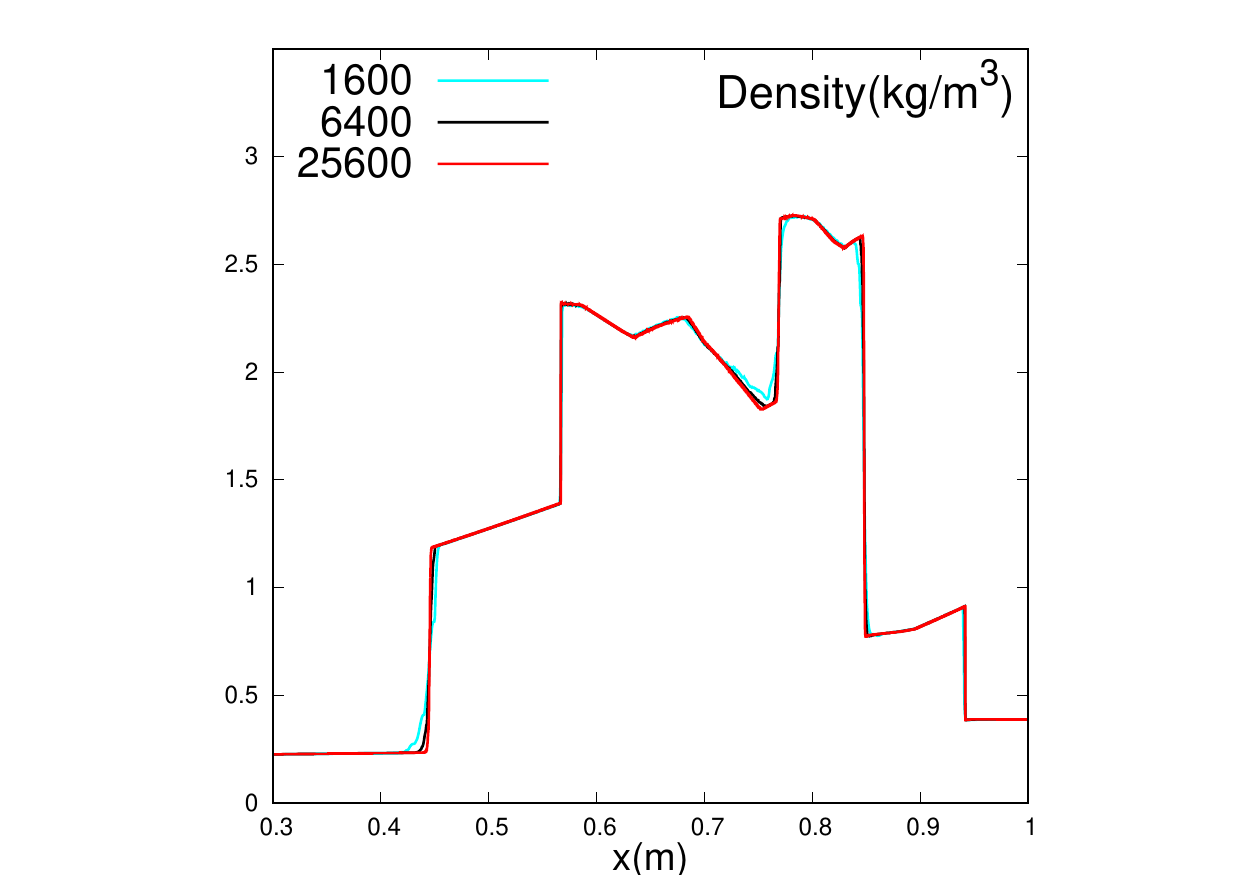}
         \includegraphics[trim={50 0 50 13},clip,width=0.35\linewidth]{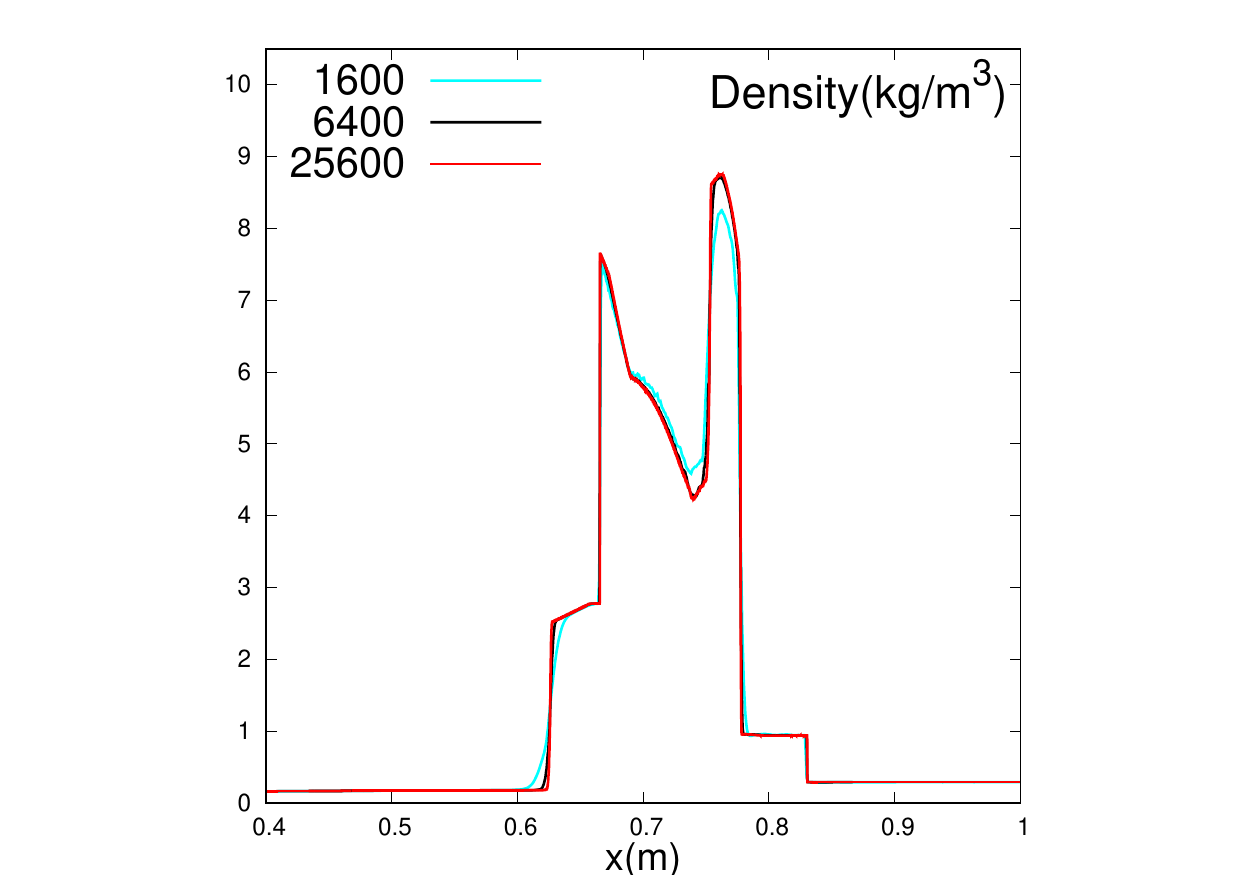}
        \caption{Density profile for the Woodward-Colella blast wave
          problem using the Jones-Wilkins-Lee equation of state.
        Left: Case~1 at $T=0.038$. Right: Case~2 at $T=\SI{8.2e-4}{}$.}
        \label{Fig:Woodward-Collela}
\end{figure}

The initial state for the blast wave problem is given as follows:
\begin{equation}
  (\rho_0(x), v_0(x), p_0(x)) = \begin{cases}
    (1, 0, 10^3)    & \quad \text{ if } x \in [0,0.1],   \\
    (1, 0, 10^{-2}) & \quad \text{ if } x \in (0.1,0.9), \\
    (1, 0, 10^2)    & \quad \text{ if } x \in [0.9,1].
  \end{cases}
\end{equation}
The simulations are performed on the domain $\Dom=(0,1)$ with the
\texttt{TAMU} code. All the tests use $\text{CFL}=0.9$ and slip boundary
conditions. We show in Figure~\ref{Fig:Woodward-Collela} the density
profiles for both cases at their respective final times  using three
different meshes composed of $1600$, $6400$, and $25600$ elements. The
results compare well with what is available in the literature; see \eg
\cite[Fig.~2]{Toro_Castro_Lee_JCP_2015} for Case~1.


\subsubsection{1D -- Riemann Problem with \texttt{SESAME} database}
We now demonstrate the method's ability to handle tabulated data as
the pressure oracle. In particular, we consider a Riemann problem
modeling the impact of a right-moving aluminum slab into a stationary
aluminum slab at high velocities. To simulate the material aluminum,
we use the Material ID \texttt{3720} in the~\texttt{SESAME}
database~\citep{lyon1992sesame}.  We let the density of the two
  aluminum slabs be $\rho_0 = \SI{3000}{kg~m^{-3}}$.  The pressure
  at this density and at room temperature ($\SI{293}{K}$) is
$p_0 = \SI{1.004489e10}{Pa}$ (this value is obtained from
the~\texttt{SESAME} database).
\begin{figure}[ht]
     \centering
         \includegraphics[trim={60 0 60 0},clip,width=0.32\linewidth]{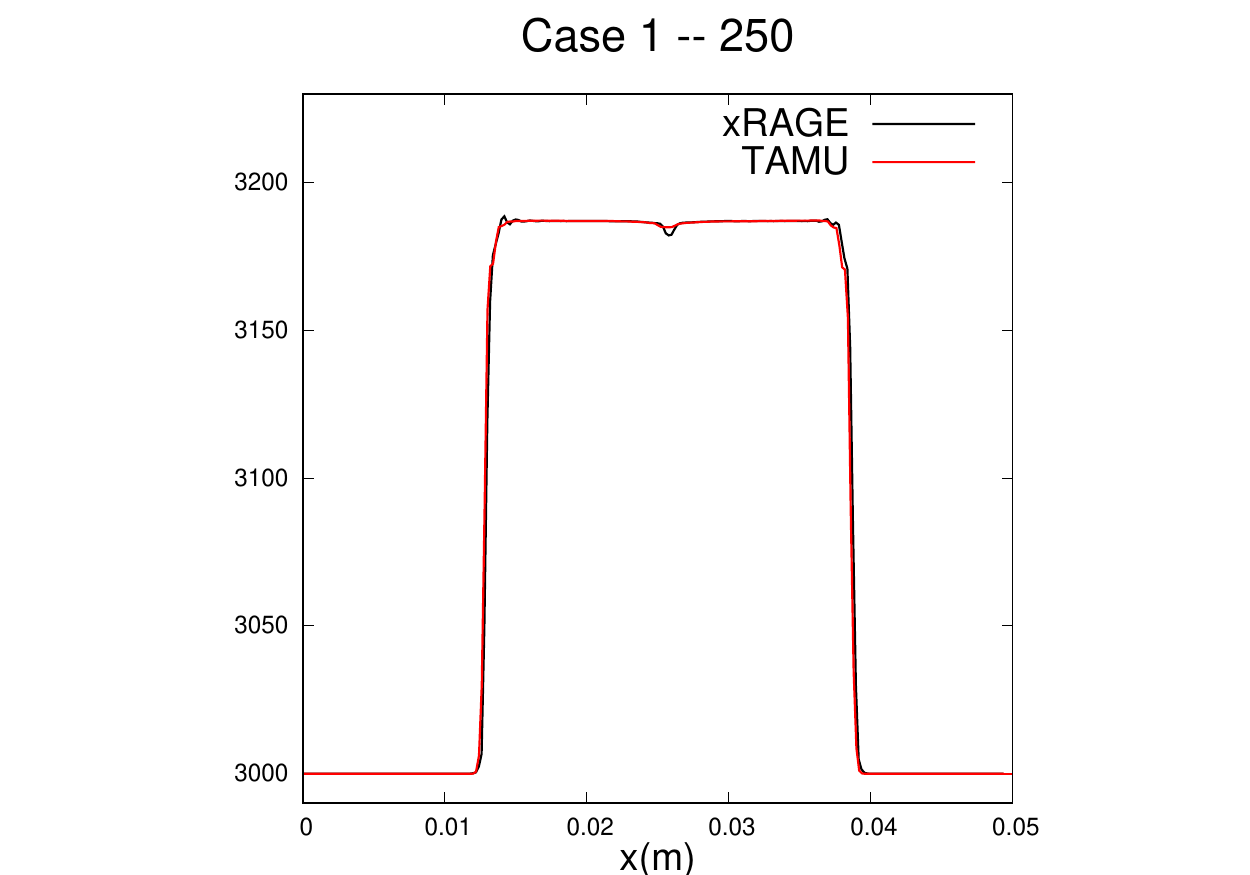}%
         \includegraphics[trim={60 0 60 0},clip,width=0.32\linewidth]{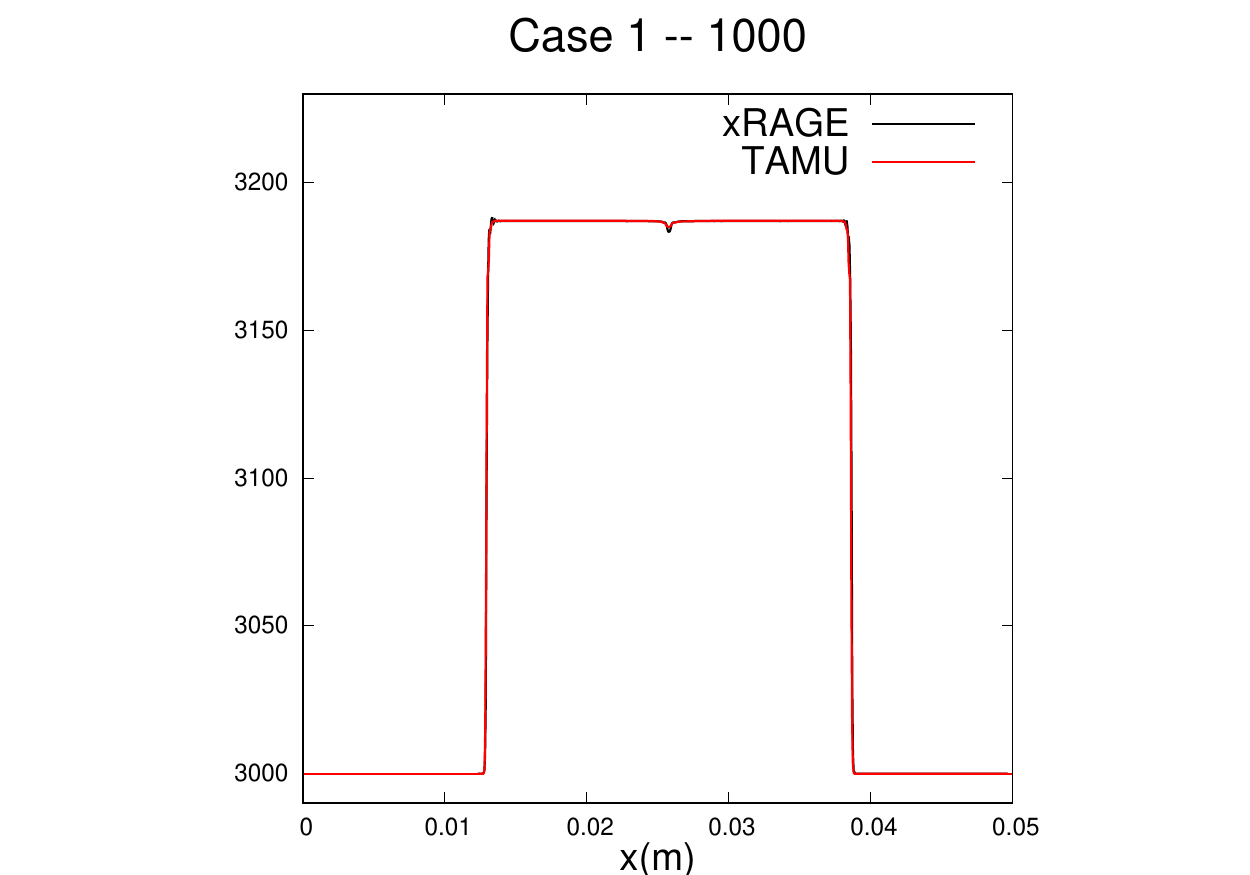}%
         \includegraphics[trim={60 0 60 0},clip,width=0.32\linewidth]{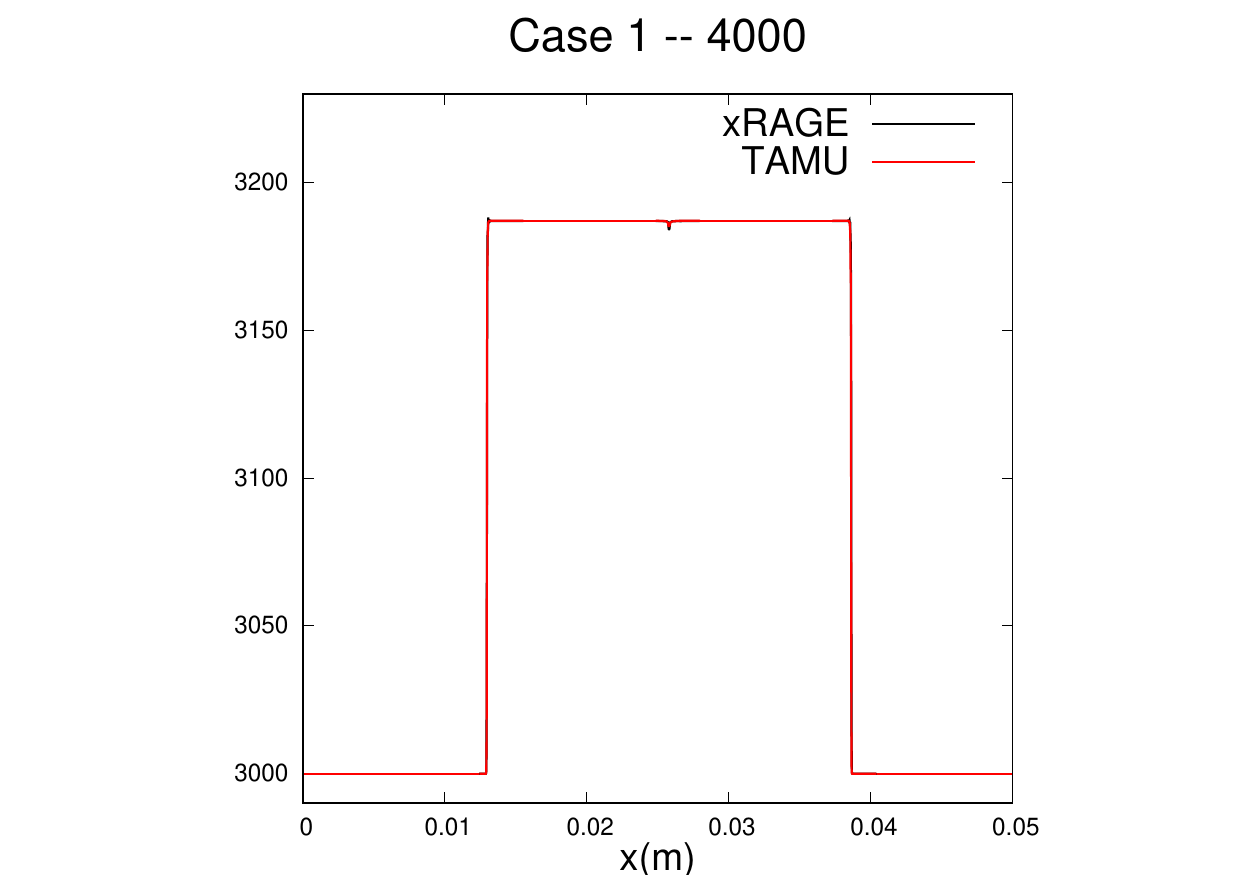}\par
         \includegraphics[trim={60 0 60 0},clip,width=0.32\linewidth]{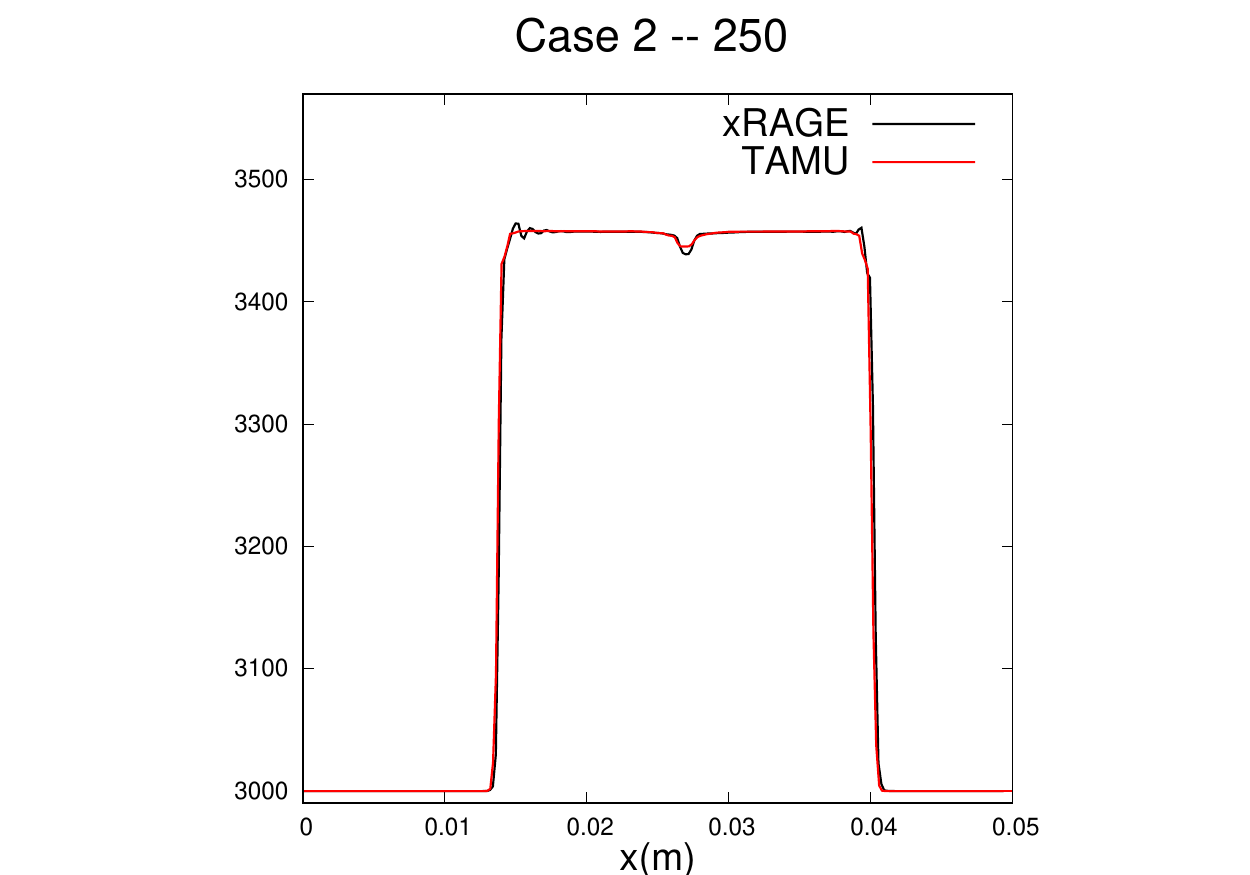}%
         \includegraphics[trim={60 0 60 0},clip,width=0.32\linewidth]{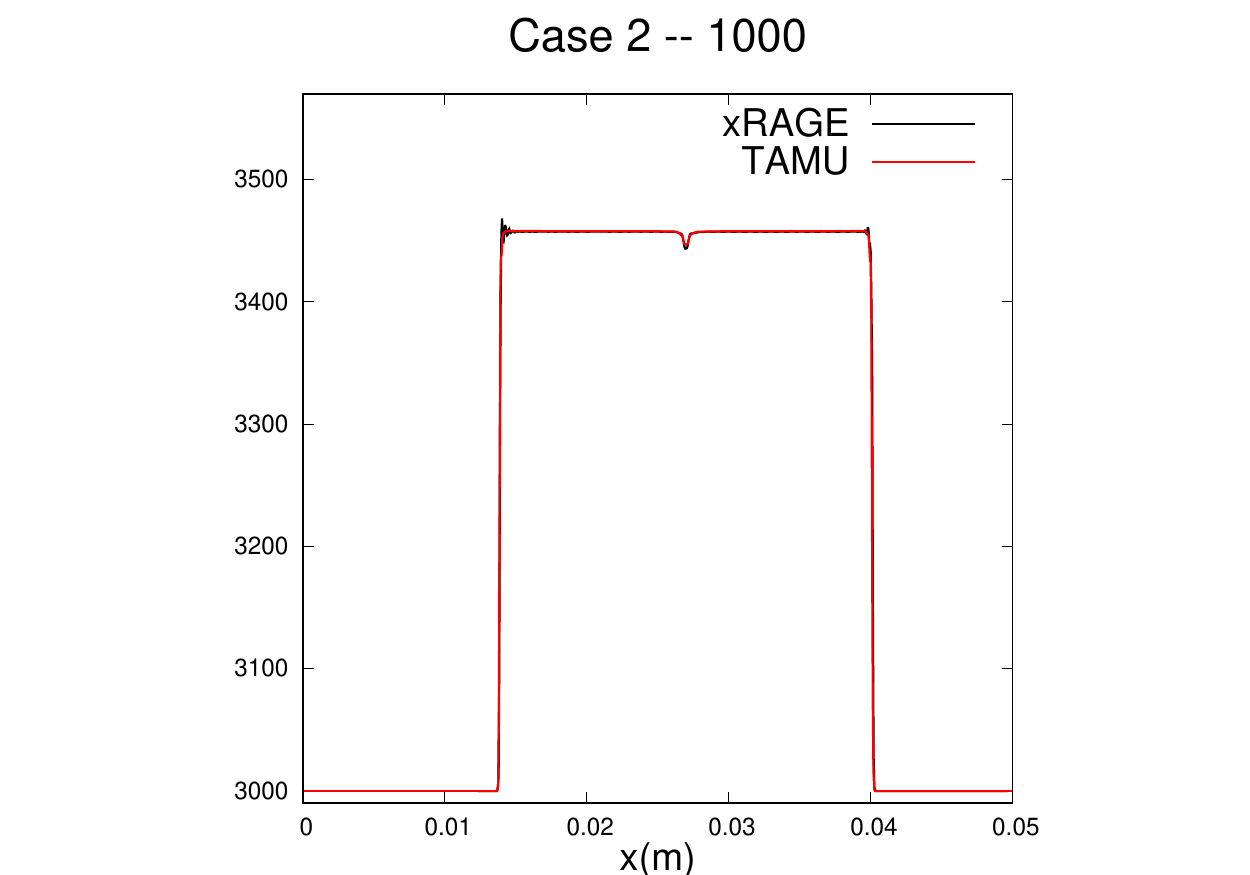}%
         \includegraphics[trim={60 0 60 0},clip,width=0.32\linewidth]{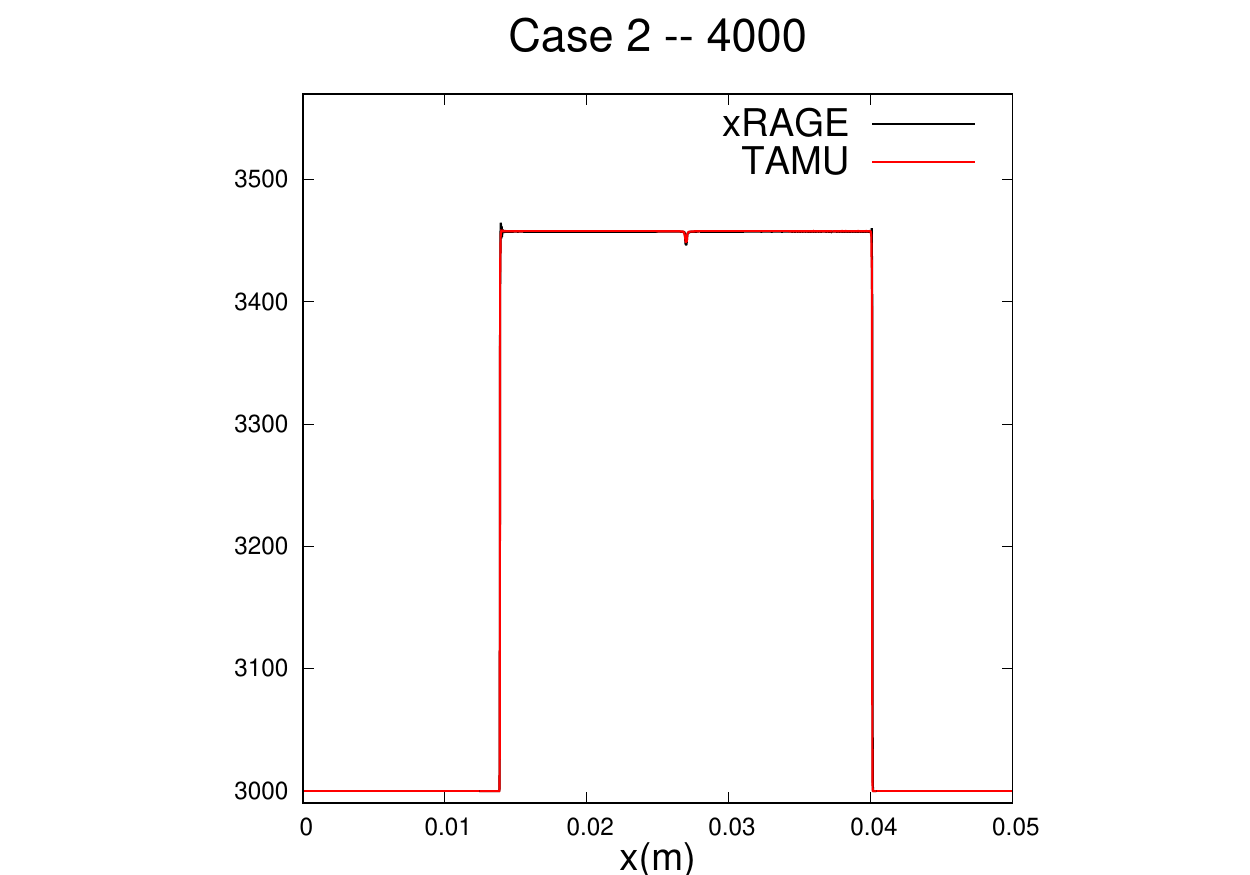}%
        \caption{Density comparison for the aluminum impact problem with the \texttt{TAMU} and \texttt{xRAGE} codes using the \texttt{SESAME} database.}
        \label{Fig:xRAGE_TAMU}
\end{figure}

The computational domain is set to $D = (0,\SI{0.05}{m})$ where the
two aluminum slabs are separated at $x = \SI{0.025}{m}$. We consider
two cases: we assume in the first case that the left slab initially
moves with velocity $\SI{800}{m~s^{-1}}$ (Case 1), and in the second
case we assume that the velocity is $\SI{2000}{m~s^{-1}}$ (Case
2). The simulations are run until final time $T = \SI{2e-6}{s}$ and
performed with 250, 1000 and 4000 mesh elements to show
convergence. We use $\text{CFL}=0.5$ and set Dirichlet conditions on
the left boundary and slip conditions on the right boundary. For
verification, we run the same configuration with the \texttt{xRAGE}
code developed at Los Alamos National Laboratory
(see:~\citep{Gittings_2008} and~\citep{grove_xrage_2019}). In
Figure~\ref{Fig:xRAGE_TAMU}, we show the density output comparison
between the two codes.We see very good agreement between the codes. This test clearly demonstrates the method's ability to handle tabulated data.

\subsection{2D -- Shock Collision with Triangular Obstacle}
We now reproduce a 2D problem proposed
in~\cite[Sec.~5.4]{Toro_Castro_Lee_JCP_2015} which investigates traveling
shock waves colliding with a triangular obstacle. This configuration is
commonly known in the literature as Schardin's
problem~\citep{schardin1957high}.  We refer the reader
to~\cite{chang2000shock} where the authors experimentally reproduced
Schardin's original experiments and give a detailed description of the
experiments.

This test is performed with the Van der Waals equation of
state~\eqref{vdw_eqn} with $\gamma = 864.7/577.8,\,a = 0.14,\, b =
\num{3.258e-5}$ and initialized as follows. The relative Mach speed is $M_S
= 1.3$ with the primitive state in front of the shock set to $\bu_R =
(1.225, 0, 0, 101325)\tr$. Using the Rankine-Hugoniot conditions to derive
the post shock state, the complete initial state is given as follows:
\begin{equation}
  \big(\rho_0(\bx), \bv_0(\bx), p_0(\bx)\big) = \begin{cases}
    (1.82039, 148.597, 0, 185145), & \quad \text{ if } x \leq -0.55, \\
    (1.225, 0, 0, 101325),         & \quad \text{ if } x > -0.55.
  \end{cases}
\end{equation}

The computations are performed with the \texttt{Ryuin} code. The
computational domain is defined as $D = (-0.65,0.5) \times (-0.5, 0.5)
{\setminus} K$ where $K$ is the triangle formed by the vertices
$(-0.2,0.0)$, $(0.1, 1/6)$, and $(0.1, -1/6)$. The simulations are run
until $T = \num{2.2e-3}$ with a CFL of $0.6$. The mesh is composed of
\SI[group-separator = {,}]{7347200}{} $\polQ_1$-nodes. Dirichlet conditions are imposed on the
left boundary, dynamic outflow conditions on the right boundary and slip
conditions on the rest of the boundaries.

We show in Figure~\ref{Fig:triangular_shock} the Schlieren output of the
shock wave interacting with the triangular obstacle at three time
snapshots $t=\{\num{1e-3},\num{1.6e-3},\num{2.2e-3}\}$. The results match
up well with the experimental photos shown in~\citep{chang2000shock}. In
particular, we see vortices developing along the slip layer behind
the back vertices of the triangle (see:~\citep[Fig.~8]{chang2000shock});
these so-called vortexlets are not apparent
in~\citep[Fig.~6]{Toro_Castro_Lee_JCP_2015} likely due to a lack of mesh
resolution.
\begin{figure}[ht]
  \centering
  \includegraphics[trim={5 0 5 0},clip,width=0.33\linewidth]{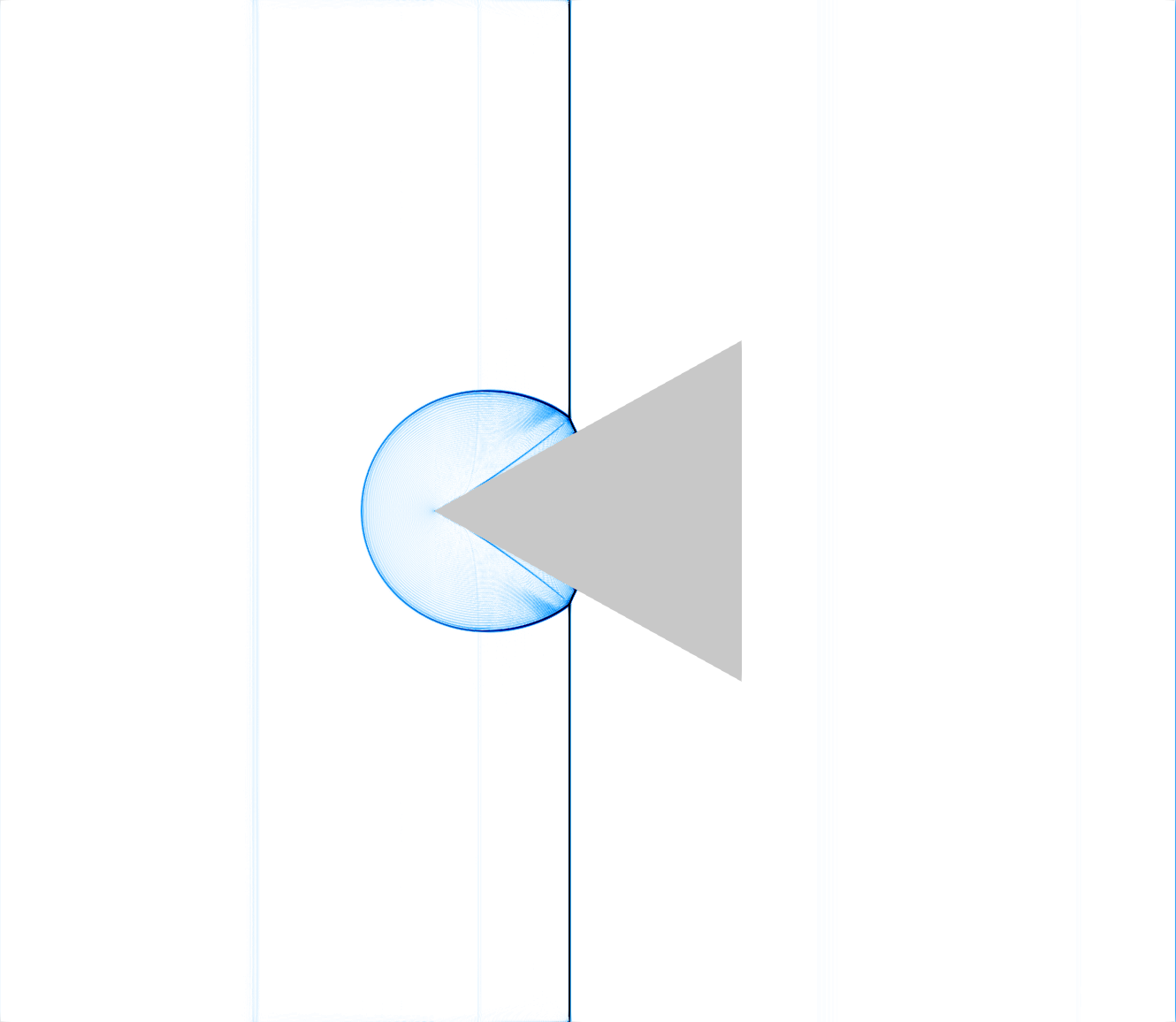}%
  \includegraphics[trim={5 0 5 0},clip,width=0.33\linewidth]{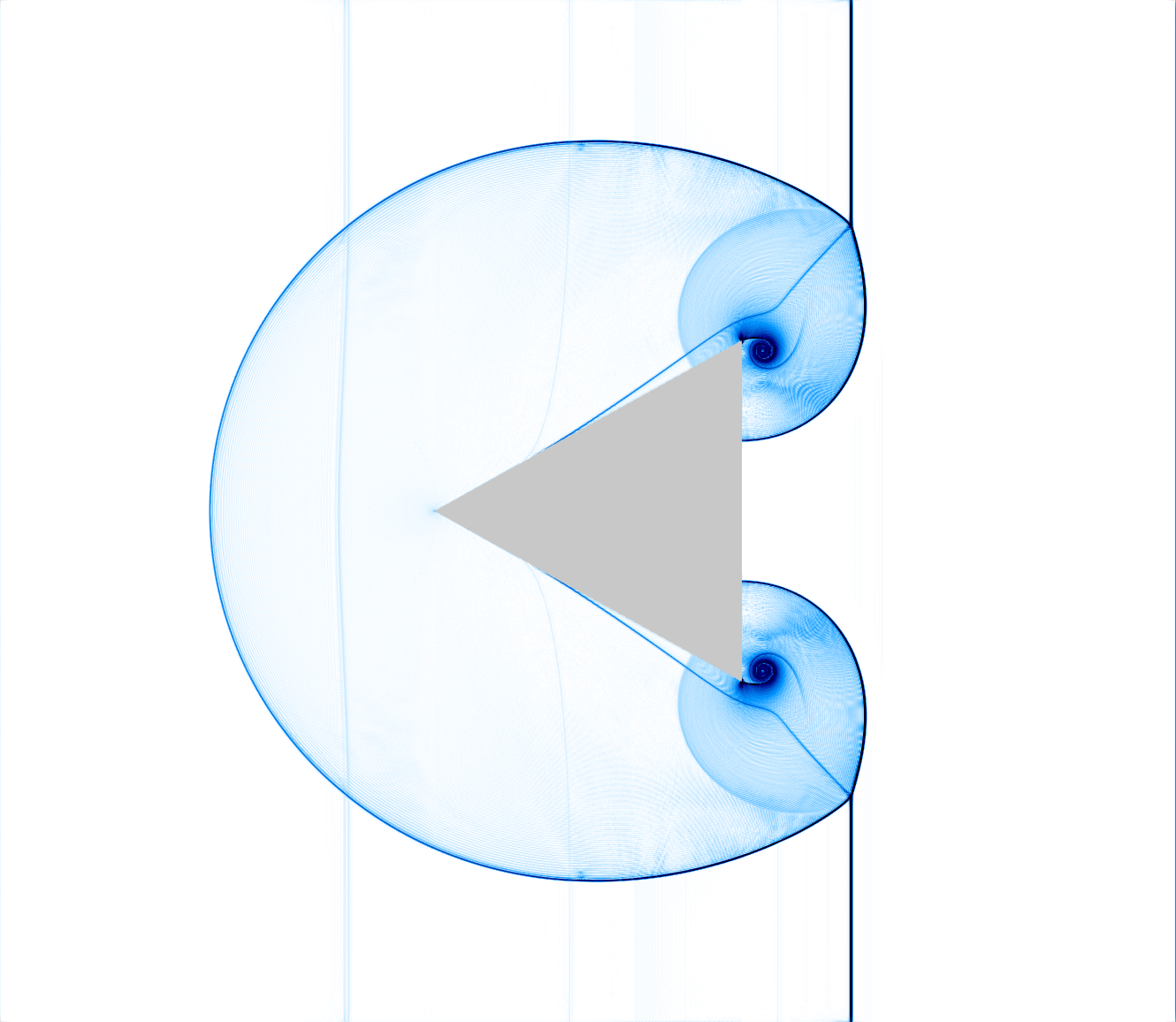}%
  \includegraphics[trim={5 0 5 0},clip,width=0.33\linewidth]{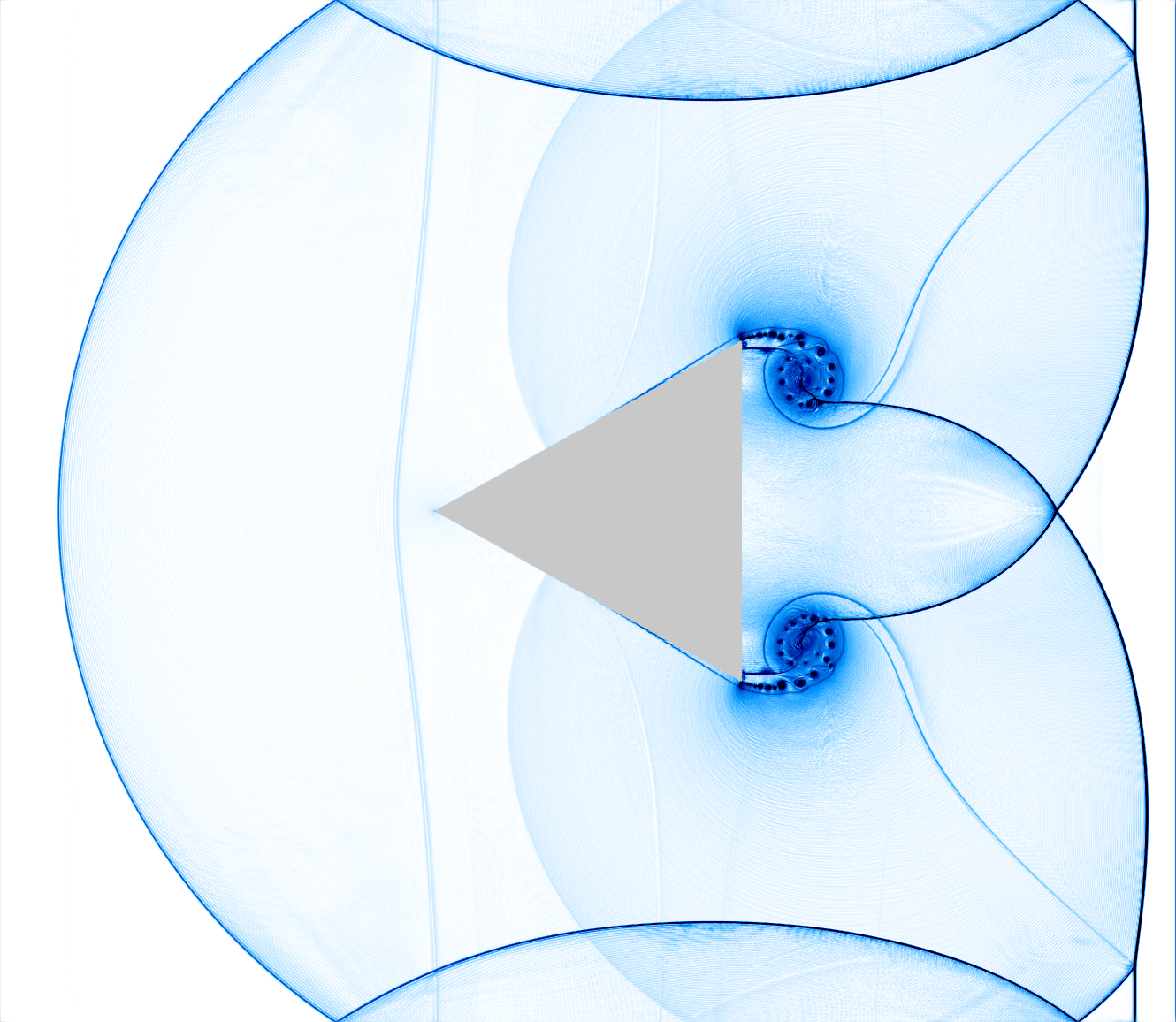}
  \caption{Shock wave interacting with triangular obstacle at $t=\{\num{1e-3},\num{1.6e-3},\num{2.2e-3}\}$.}
  \label{Fig:triangular_shock}
\end{figure}

\subsection{2D -- Shock Bubble Interaction}

In this section, we consider a single material shock-bubble
interaction benchmark proposed in~\cite[Sec~5.2.2]{wang2021stiffened}
using the Jones-Wilkins-Lee equation of state. For more details on
shock-bubble interaction problems, we refer to
\cite{haas_sturtevant_1987} for experimental results and to
\cite{quirk_karni_1996} for the description of the simulation
setup. This test demonstrates the robustness of our method by being
able to reproduce the complex interactions of the shock hitting the
bubble.

Let $\mathfrak{B}$ denote the bubble
centered at $(\SI{0.5}{},\SI{0.5}{})$ with radius $\SI{0.25}{}$. The
primitive states at the initial time for the ambient fluid and bubble are
respectively defined as follows:
\begin{align}
  (\rho_R, \bv_R, p_R) & = (\SI{1.0e3}{}, \SI{0}{}, \SI{5.0e10}{}), \\
  (\rho_\mathfrak{B}, \bv_\mathfrak{B}, p_\mathfrak{B}) & = (\SI{2.0e3}{}, \SI{0}{}, \SI{5.0e10}{}).
\end{align}
The pressure across the shock is prescribed to be $p_L = \SI{4.369e11}{}$,
and the remaining state variables, $\rho_L$ and $\bv_L$, are computed
using the Rankine-Hugoniot conditions.  Thus, the initial primitive state
for the problem is given by
\begin{equation}
  (\rho_0(\bx), \bv_0(\bx), p_0(\bx)) = \begin{cases}
    (3778.85, 16867.6, 0, 4.369\times 10^{11}), & \text{ if } x < 0.05,                                          \\
    (1000, 0, 0, 5\times 10^{10}),                            & \text{ if } x \geq 0.05 \text{ and } \bx \not\in \mathfrak{B}, \\
    (2000, 0, 0, 5\times 10^{10}),                              & \text{ if } \bx \in \mathfrak{B}.
\end{cases}
\end{equation}
We perform the tests with the \texttt{Ryujin} code. The parameters for the
JWL EOS~\eqref{jwl_eqn} are set to
\begin{align*}
  A = \SI{8.545e11}{},\,B = \SI{2.05e10}{},\,R_1 = 4.6,\,
  R_2 = 1.35,\,\omega = 0.25,\,\rho_0 = \SI{1.84e3}{}.
\end{align*}
The computational domain is set to $D=(0, \SI{3}{})\times(0, \SI{1}{})$ and
the mesh is composed of \SI[group-separator = {,}]{50348033}{} $\polQ_1$-nodes. The simulations
are performed until the final time $T = \SI{100}{\mu\s}$ with CFL$ = 0.5$.
Slip boundary conditions are applied on all boundaries. In
Figures~\ref{fig1:wang_li_schlieren_shock_bubble} to \ref{fig3:wang_li_schlieren_shock_bubble}, we show the Schlieren
plots of the density at times $t = \{\SI{40}{\mu\s}$, \SI{70}{\mu\s}$,
\SI{100}{\mu\s}\}$. The results compare well with Figures 5.6 and 5.7
shown in~\cite{wang2021stiffened}.
\begin{figure}[ht]
  \centering
  \includegraphics[scale=0.12]{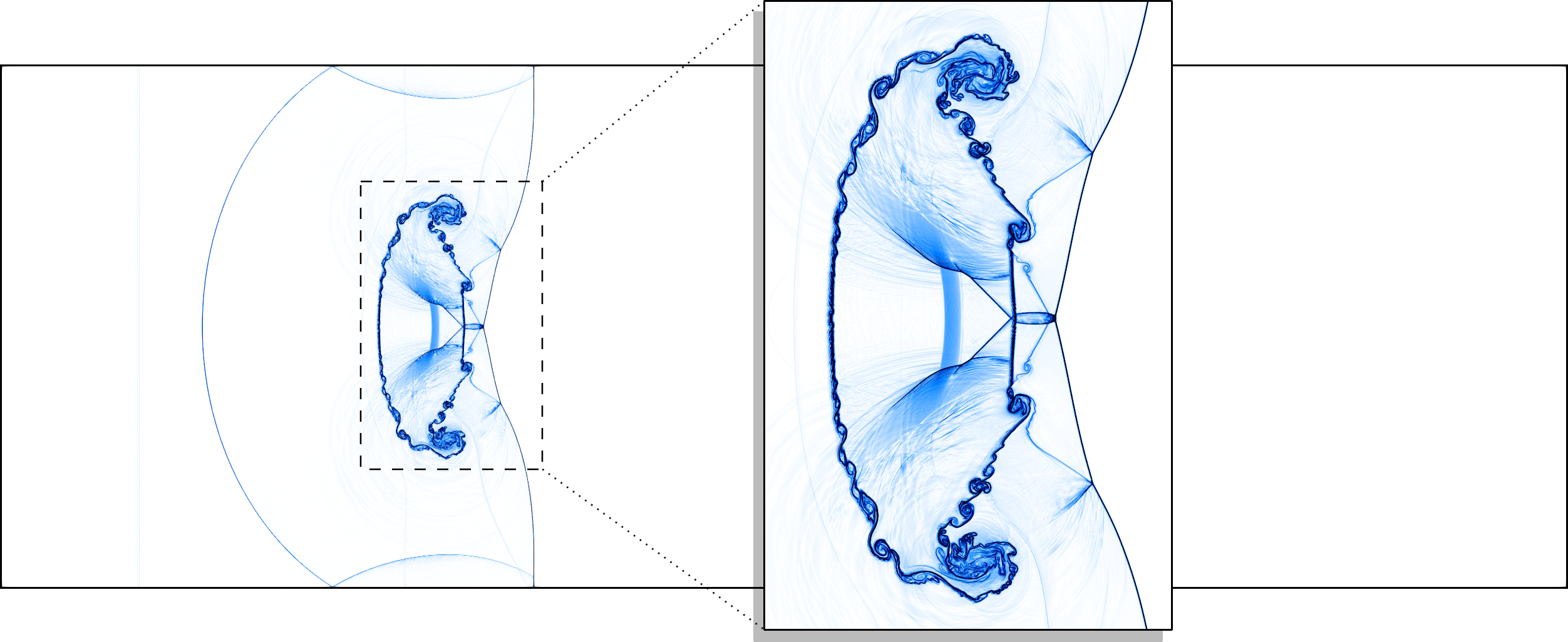}
  \caption{Schlieren plots for the shock-bubble interaction benchmark for
    $t = \SI{40}{\mu\s}$.}
  \label{fig1:wang_li_schlieren_shock_bubble}
\end{figure}
\begin{figure}[ht]
  \centering
  \includegraphics[scale=0.12]{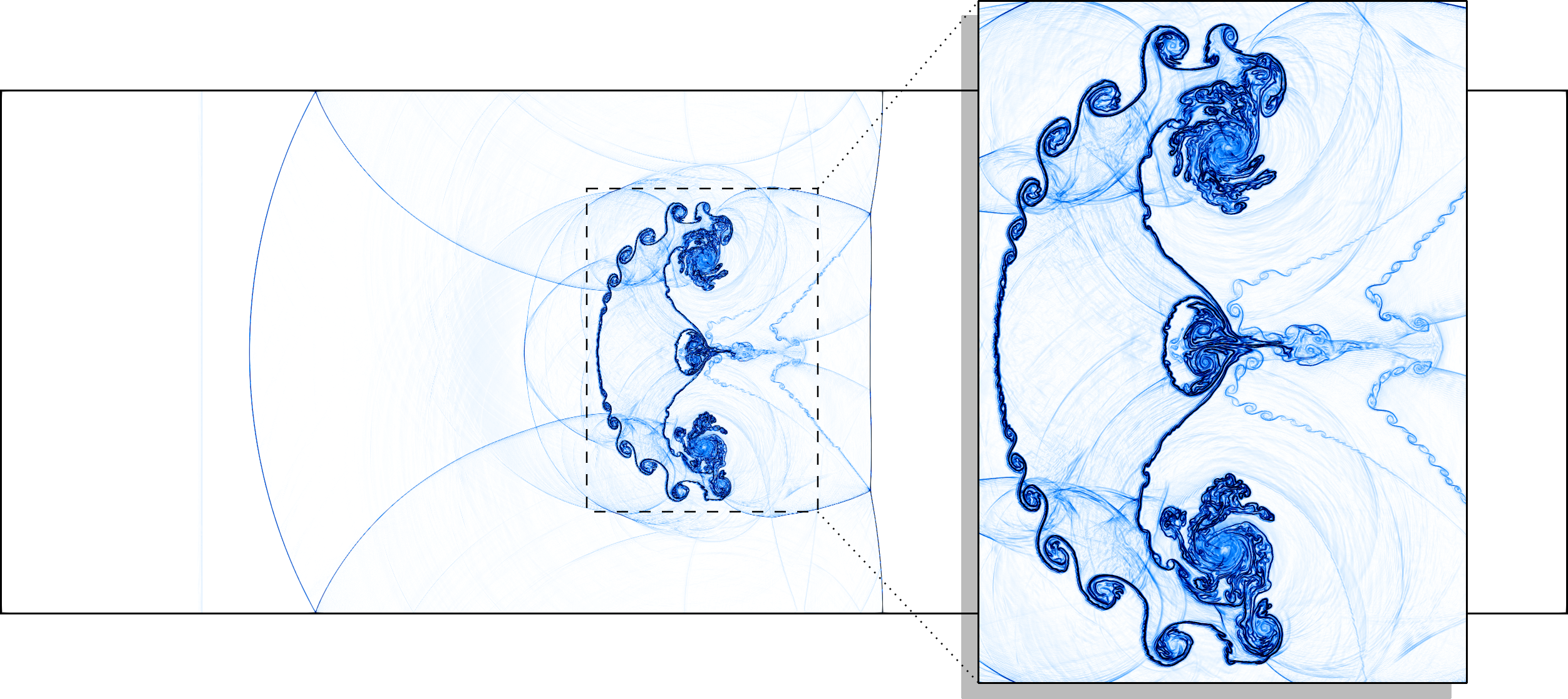}\par
  \caption{Schlieren plots for the shock-bubble interaction benchmark for
$t = \SI{70}{\mu\s}$.}
  \label{fig2:wang_li_schlieren_shock_bubble}
\end{figure}
\begin{figure}[ht]
  \centering
  \includegraphics[scale=0.12]{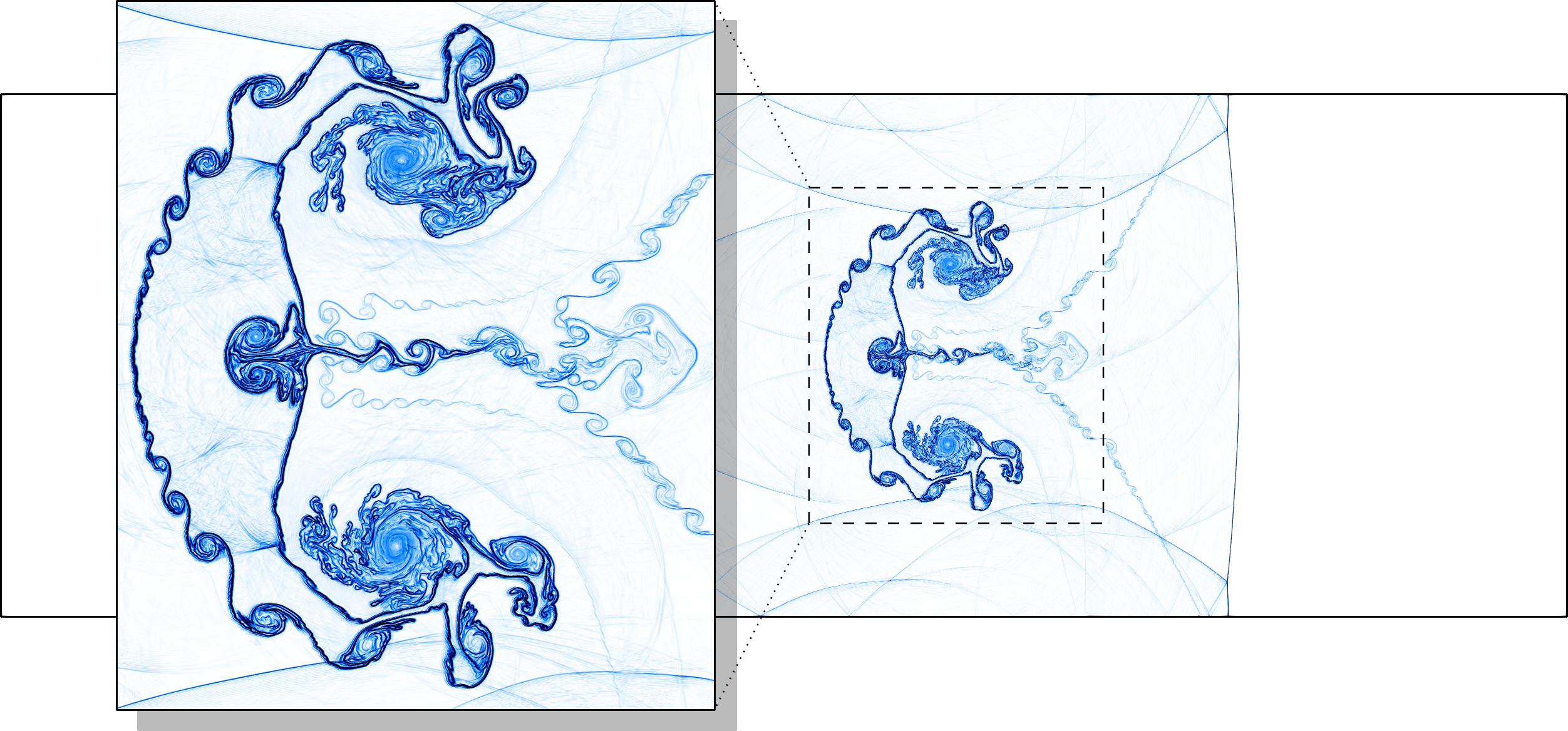}
  \caption{Schlieren plots for the shock-bubble interaction benchmark for
$t = \SI{100}{\mu\s}$ (bottom).}
  \label{fig3:wang_li_schlieren_shock_bubble}
\end{figure}


\section{Conclusion}
\label{sec:conclusion}

We have developed an invariant-domain preserving second-order accurate
method for the Euler equations with arbitrary or tabulated equation of
states. The work presented is the continuation and extension of
\cite{Guermond_Nazarov_Popov_Tomas_SISC_2019} and~\cite{CGP_2022}. We
proposed a surrogate entropy functional that increases across shocks in the
associated 1D Riemann problem to work around the lack of a general
mathematical entropy valid for arbitrary equations of state. A convex
limiting procedure was performed on this surrogate entropy functional to
enforce a local minimum principle. This in turn implies a positive local
lower bound on the internal energy. Numerical evidence of higher order
accuracy was demonstrated with convergence tests and several computational
benchmarks.


\appendix
\section{Isentropic vortex with van der Waals equation of state}
\label{appendix:isen_vortex_vdw}
We present a derivation of the isentropic vortex solution with the van der
Waals equation of state and give some necessary conditions for the
existence of this solution.
\begin{theorem}
  Consider the van der Waals equation of state~\eqref{vdw_eqn} with
  $a > 0$, $b \eqq 0$, and $\gamma \eqq \frac32$ or $\gamma \eqq
  2$. Let $\bx^0 \in \Real^2$, $\beta > 0$, $r_0 > 0$. Let
  $\rho_\infty>0$, $\bv_\infty\in\Real^2$, $p_\infty>0$ and assume
  that
  \begin{equation}
      p_\infty > \tfrac13 a \rho_\infty^2,\qquad
      a \rho_\infty + \frac{3 p_\infty}{\rho_\infty} >
      \frac{\beta^2 e^1}{8 r_0^2 \pi^2}
  \end{equation}
 if $\gamma = \tfrac32$.
  The following density, velocity, and pressure fields solve the
  compressible Euler equations~\eqref{mass}--\eqref{total_energy} with the
  van der Waals equation of state:
  \begin{subequations}
    \begin{align}
      \rho(\bx,t) & \eqq
      \begin{cases}
        \label{isentropic_vdw_density}
        \Big( \tfrac{3C}{4a} - \tfrac{1}{2}\sqrt{\tfrac{9C^2}{4a^2} +
        \tfrac{2}{a}\big( F + \tfrac{1}{2r_0^2} \psi(\overline{\bx})^2
        \big)} \Big)^2, & \; \text{ if } \gamma = \tfrac32,
        \\
        \rho_\infty - \frac{\rho^2_\infty}{4 p_\infty r_0^2}
        \psi(\overline{\bx}), & \; \text{ if } \gamma = 2,
      \end{cases}
      \\
      \bv(\bx,t) & \eqq \bv_\infty + \psi(\overline{\bx}) (-\bar{x}_2,
      \bar{x}_1)\tr,
      \\
      p(\bx,t) & \eqq C \rho(\bx,t)^\gamma - a \rho(\bx,t)^2,
      \label{isentropic_vdw_pressure}
    \end{align}
  \end{subequations}
  with $\psi(\overline{\bx}) \eqq \tfrac{\beta}{2\pi}\exp(\tfrac12(1 -
  \tfrac{1}{r_0^2}\Vert\overline{\bx}\Vert^2_{\ell^2}))$, $(\bar{x}_1,
  \bar{x}_2) = \overline{\bx} \eqq \bx - \bx^0 - \bv_\infty t$, $C =
  (p_\infty + a \rho_\infty^2)/\rho_\infty^{3/2}$, and $F =
  -a\rho_\infty - 3p_\infty/\rho_\infty$.
\end{theorem}
\begin{proof}
  The derivation of the isentropic vortex begins with the additional
  assumption that the velocity field is divergence free.  That is, $\DIV\bv
  = 0$.  Under this assumption the Euler equations take the following
  simplified form:
  \begin{align}
    \partial_t \rho(\bx,t) + \bv (\bx,t)\SCAL\GRAD\rho
    (\bx,t)
    & = 0,                                      &  & \quad \bx \in \mathbb{R}^2, \, t > 0,                       \\
    \partial_t \bv(\bx,t)+ (\bv(\bx,t)\SCAL\GRAD ) \bv(\bx,t)
    & = -\frac{1}{\rho(\bx,t)} \nabla p(\bx,t), &  & \quad \bx \in \mathbb{R}^2, \, t > 0 \label{isenvortex_eq2} \\
    \partial_t e(\bx,t) + \bv(\bx,t) \SCAL\GRAD e(\bx,t)
    & = 0,                                      &  & \quad \bx \in \mathbb{R}^2, \, t > 0,
  \end{align}
  with $\bx \eqq (x_1, x_2)$, boundary conditions, $(\rho_\infty,
  \bv_\infty {\eqq}(v_{1,\infty}, v_{2,\infty})\tr, p_\infty)$ and yet to
  be determined initial conditions $(\rho_0(\bx), \bv_0(\bx),
  p_0(\bx))$.  To keep things general, we make no assumption on the
  equation of state for $p = p(\rho,e)$.

  We write the solution as a perturbation of the far-field state; \ie we
  define $\bv \eqq \bv_\infty + \delta \bv$ with
  \begin{equation}
    \delta \bv(\bx,t) :=
    \begin{pmatrix}
      \partial_{x_2} \psi(\bx - \bx^0 - \bv_\infty t) \\
      -\partial_{x_1} \psi(\bx - \bx^0 - \bv_\infty t)
    \end{pmatrix},
  \end{equation}
  with the stream function $\psi(\bx) :=
  \frac{\beta}{2\pi}\exp(\tfrac{1}{2}(1 -
  \tfrac{\Vert\bx\Vert_{\ell^2}^2}{r_0^2}))$. Here $\bx^0 := (x_1^0, x_2^0)
  \in \mathbb{R}^2$, $\beta$, and $r_0$ are free parameters.  To further
  simplify notation, define $(\bar{x}_1, \bar{x}_2) = \overline{\bx} := \bx
  - \bx^0 - \bv_\infty t$ and $r^2 := \Vert\overline{\bx}\Vert_{\ell^2}^2$.
  Note the following identities which will be used later on:
  \begin{align}
    \partial_{x_i} \psi(\overline{\bx})
    & =
    -\frac{\bar{x}_i}{r_0^2} \psi(\overline{\bx}),
    \label{isen_gauss_identity1}
    \\
    \partial_{x_ix_j} \psi(\overline{\bx})
    & =
    \frac{1}{r_0^2} \Big( -\delta_{ij} + \frac{\bar{x}_i \bar{x}_j}{r_0^2}
    \Big) \psi(\overline{\bx}), \label{isen_gauss_identity2}
    \\
    \partial_{t x_i} \psi(\overline{\bx})
    & =
    \frac{1}{r_0^2} \Big(v_{i,\infty} - \frac{\bar{x}_i \bv_\infty \cdot
    \overline{\bx}}{r_0^2} \Big) \psi(\overline{\bx}),
    \label{isen_gauss_identity3}
  \end{align}
  where $\delta_{ij}$ is the Kronecker symbol and $i,j \in \{1, 2\}$.

  Using that $\bv = \bv_\infty + \delta \bv$, the left hand side of
  \eqref{isenvortex_eq2} becomes,
  \begin{equation}
    \partial_t \bv + \bv \cdot \nabla \bv =  \partial_t (\delta \bv)
    + (\bv_\infty \cdot \nabla) \delta \bv + (\delta \bv \cdot \nabla) \delta \bv.
  \end{equation}
  From the definition of $\delta \bv$ and the identities
  \eqref{isen_gauss_identity1}, \eqref{isen_gauss_identity2} and
  \eqref{isen_gauss_identity3}, we have,
  \begin{align}
    (\delta \bv \cdot \nabla) \delta \bv & =
    \begin{bmatrix}
      (\partial_{x_2} \psi) (\partial^2_{x_1 x_2} \psi) - (\partial_{x_1}
      \psi) (\partial^2_{x_2} \psi)
      \\
      -(\partial_{x_2} \psi) (\partial^2_{x_1} \psi) + (\partial_{x_1}
      \psi) (\partial^2_{x_1 x_2} \psi)
    \end{bmatrix}
    = - \frac{\overline{\bx}}{r_0^4}\psi(\overline{\bx})^2
    \\
    (\bv_\infty \cdot \nabla) \delta \bv & = \frac{1}{r_0^2}\Big(-
    \begin{bmatrix} v_{2,\infty} \\ -v_{1,\infty} \end{bmatrix}
    + \frac{\bv_\infty \cdot \overline{\bx}}{r_0^2}
    \begin{bmatrix} \bar{x}_2 \\ -\bar{x}_1 \end{bmatrix}\Big) \psi(\overline{\bx})
    \\
    \partial_t (\delta \bv) & = \frac{1}{r_0^2}\Big(
    \begin{bmatrix} v_{2,\infty} \\ -v_{1,\infty} \end{bmatrix}
    - \frac{\bv_\infty \cdot \overline{\bx}}{r_0^2}
    \begin{bmatrix} \bar{x}_2 \\ -\bar{x}_1 \end{bmatrix} \Big)
    \psi(\overline{\bx})
  \end{align}
  Thus equation \eqref{isenvortex_eq2} becomes
 $-\frac{\overline{\bx}}{r_0^4} \psi(\overline{\bx})^2 =
    -\frac{1}{\rho} \nabla p$. This identity is furthermore written as,
  \begin{equation} \label{diff_eq_for_psi} -\frac{1}{2r_0^2}
    \nabla(\psi(\overline{\bx})^2) = \frac{1}{\rho(t,\bx)} \nabla
    p(\rho(t,\bx)).
  \end{equation}

  Up to this point, we have not made any assumption on the equation of
  state. We recover the well known isentropic vortex solution  if we
  assume the pressure is given by the ideal gas law; \ie $p(\rho) =
  C\rho^\gamma$ for the isentropic flow where $C =
  p_\infty/\rho_\infty^\gamma$. We now proceed with the van der Waals
  equation of state. For isentropic flows we have
  \begin{equation} \label{isentropic_vdw_eos} p(\rho) =
    \frac{C\rho^\gamma}{(1 - b\rho)^{\gamma}} - a\rho^2,
  \end{equation}
  where $C$ is some constant.  (Note, we work with an arbitrary $b$ to keep
  things general in the beginning.)  Following the same process as in the
  ideal gas case, we compute the indefinite integral, $\int
  \frac{1}{\rho} \partial_{x_i} p(\rho) \: dx_i$:
  \begin{multline*}
    -\frac{1}{2r_0^2}\int \partial_{x_i} \psi(\overline\bx)^2 \: dx_i =
    \int \frac{1}{\rho} \partial_{x_i} p(\rho) \: dx_i
    = \frac{p(\rho)}{\rho} + \int \frac{p(\rho)}{\rho^2}
    \partial_{x_i} \rho\: dx_i   \\
    = \frac{p(\rho)}{\rho} + \int \Big(\frac{C\rho^{\gamma-2}}{(1
    - b\rho)^{\gamma}} - a\Big) \rho_{x_i} dx_i
    = \frac{p(\rho)}{\rho} + \int \frac{\partial}{\partial x_i} \Big[
    \frac{C}{\gamma-1} \Big( \frac{\rho}{1 - b\rho}
    \Big)^{\gamma-1} - a\rho \Big] \: dx_i
    \\
    = \frac{C \rho^{\gamma-1} (\gamma - b\rho)}{(\gamma-1)(1 -
    b\rho)^\gamma} - 2 a \rho + F.
  \end{multline*}
  Hence, $\rho(\overline\bx)$ can be found by solving the equation,
  \begin{equation}\label{nonlin_eq_for_finding_rho}
    -\frac{1}{2r_0^2} \psi(\overline{\bx})^2= \frac{C \rho^{\gamma-1}
    (\gamma - b\rho)}{(\gamma-1)(1 - b\rho)^\gamma} - 2 a \rho + F.
  \end{equation}
  We have two immediate cases for solutions that can be found explicitly.

  \textbf{Case 1: $\gamma = 3/2$ and $b=0$:}
  In this case, \eqref{nonlin_eq_for_finding_rho} becomes a quadratic
  equation for $\sqrt{\rho}$,
  \begin{equation} \label{isen_vort_nonlinear_eq} \rho -
    \frac{3C}{2a} \sqrt{\rho} - \frac{1}{2a} \Big(F +
    \frac{1}{2r_0^2} \psi(\overline{\bx})^2\Big) = 0.
  \end{equation}
  The constants $C$ and $F$ are
  determined by applying the far field condition to
  \eqref{isentropic_vdw_eos} and \eqref{isen_vort_nonlinear_eq}:
  \begin{equation} C = \frac{p_\infty +
    a\rho_\infty^2}{\rho^{3/2}_\infty} \quad \text{and} \quad F = -
    a \rho_\infty - \frac{3 p_\infty}{\rho_\infty}.
  \end{equation}
  However, care must be taken in the choice of $p_\infty$ and
  $\rho_\infty$ so that the sound speed remains real. Recall that the
  sound speed for the van der Waals EOS is
  $ c(\rho,p) = \sqrt{\gamma \frac{p + a \rho^2}{\rho(1 - b\rho)} -
    2a\rho}$.
  The hypothesis $p_\infty > \frac{1}{3} a \rho_\infty^2$ guarantees
  that $c(\rho_\infty, p_\infty)^2 > 0$.

  The physical root for equation \eqref{isen_vort_nonlinear_eq} is
  $\sqrt{\rho} = \tfrac{3C}{4a} - \tfrac{1}{2}\sqrt{\tfrac{9C^2}{4a^2}
    + \tfrac{2}{a}\big(F + \tfrac{1}{2r_0^2} \psi(\overline{\bx})^2
    \big)}$.
  Furthermore, for the root to be real we require that $-F >
  \frac{1}{2r_0^2}\psi(\overline{\bx})^2$ for all $\overline{\bx} \in
  \mathbb{R}^2$.  In particular,
  \begin{equation} a\rho_\infty +
    \frac{3p_\infty}{\rho_\infty} > \frac{\beta^2 e^1}{8r_0^2 \pi^2}.
  \end{equation}

  Lastly, we must justify that the system remains hyperbolic; that is, the
  sound speed is real for all $(\bx, t) \in \Real^2 \CROSS [0,\infty)$.
  Since the flow is isentropic, the sound speed for the van der Waals EOS
  (with $\gamma = 3/2$ and $b = 0$) is,
 $f(\rho) := c(p(\rho),\rho)^2 = \frac32 C \sqrt{\rho} - 2a\rho$.
  Note that $\lim_{\rho \to 0^+} f(\rho) = 0$, $f'(\rho) =
  \frac{3C}{4\sqrt{\rho}} - 2a$, and $\lim_{\rho \to 0^+} f'(\rho)
  = \infty$.  Therefore, $f(\rho)$ has a maximum at $\rho = \big(
  \frac{3C}{8a} \big)^2$ and hence $f(\rho) > 0$ for $\rho \in (0,
  \big( \frac{3C}{8a} \big)^2)$.  From the definition of $\rho$,
  \eqref{isentropic_vdw_density}, we see that $0 < \rho <
  \big(\frac{3C}{4a}\big)^2$.  Thus the sound speed is always real.

  \textbf{Case 2: $\gamma = 2$ and $b = 0$:}
  For these choices of parameters, \eqref{nonlin_eq_for_finding_rho}
  becomes,
  \begin{equation}\label{gamma_2_linear_eq}
      2(C - a)\rho + F + \frac{1}{2r_0^2} \psi(\overline{\bx}) = 0.
  \end{equation}
  Using the far field boundary conditions for
  \eqref{isentropic_vdw_pressure} and \eqref{gamma_2_linear_eq} we find
  that $C = \frac{p_\infty}{\rho^2_\infty} + a$ and $F = -2
  p_\infty/\rho_\infty$, respectively. Solving for $\rho$ in
  \eqref{gamma_2_linear_eq} we have,
  \begin{equation}
      \rho = \rho_\infty - \frac{\rho^2_\infty}{4 p_\infty r_0^2}
      \psi(\overline{\bx}).
  \end{equation}
  Note the sound speed is $c(p(\rho),\rho)^2 = 2(C-a)\rho
  = \frac{2p_\infty}{\rho_\infty} \rho > 0$.
\end{proof}

\bibliographystyle{abbrvnat}
\bibliography{ref_generic_eos}
\end{document}